\documentclass[12pt]{amsart}
\usepackage{amsmath,amssymb,amsbsy,amsfonts,latexsym,amsopn,amstext,
                                               amsxtra,euscript,amscd,bm}
\usepackage{url}
\usepackage[colorlinks,linkcolor=blue,anchorcolor=blue,citecolor=blue]{hyperref}
\usepackage{color}
\usepackage{float}
\hypersetup{breaklinks=true}
\usepackage{multirow}

\begin{document}

\newtheorem{theorem}{Theorem}
\newtheorem{lemma}[theorem]{Lemma}
\newtheorem{claim}[theorem]{Claim}
\newtheorem{cor}[theorem]{Corollary}
\newtheorem{prop}[theorem]{Proposition}
\newtheorem{example}[theorem]{Example}
\newtheorem{definition}{Definition}
\newtheorem{quest}[theorem]{Question}
\newcommand{\hh}{{{\mathrm h}}}

%\numberwithin{equation}{section}
%\numberwithin{theorem}{section}
%\numberwithin{table}{section}

\def\sssum{\mathop{\sum\!\sum\!\sum}}
\def\ssum{\mathop{\sum\ldots \sum}}
\def\iint{\mathop{\int\ldots \int}}

\def\squareforqed{\hbox{\rlap{$\sqcap$}$\sqcup$}}
\def\qed{\ifmmode\squareforqed\else{\unskip\nobreak\hfil
\penalty50\hskip1em\null\nobreak\hfil\squareforqed
\parfillskip=0pt\finalhyphendemerits=0\endgraf}\fi}%%

%  use the AMS-Euler Fraktur fonts
%%%%%%%%%%%%%%%%%%%%%%%%%%%%%%%%%%
\newfont{\teneufm}{eufm10}
\newfont{\seveneufm}{eufm7}
\newfont{\fiveeufm}{eufm5}
%%%%%%%%%%%%%%%%%%%%%%%%%%%%%%%%%
%
%  allow automatic size selection in math mode
%
%%%%%%%%%%%%%%%%%%%%%%%%%%%%%%%%%
\newfam\eufmfam
     \textfont\eufmfam=\teneufm
\scriptfont\eufmfam=\seveneufm
     \scriptscriptfont\eufmfam=\fiveeufm
%%%%%%%%%%%%%%%%%%%%%%%%%%%%%%%%%
%
%  \frak works on a single symbol at a time...
%
\def\frak#1{{\fam\eufmfam\relax#1}}

\numberwithin{equation}{section}
\numberwithin{theorem}{section}
\numberwithin{table}{section}

\newcommand{\bflambda}{{\boldsymbol{\lambda}}}
\newcommand{\bfmu}{{\boldsymbol{\mu}}}
\newcommand{\bfxi}{{\boldsymbol{\xi}}}
\newcommand{\bfrho}{{\boldsymbol{\rho}}}

\def\fK{\mathfrak K}
\def\fT{\mathfrak{T}}

\def\fA{{\mathfrak A}}
\def\fB{{\mathfrak B}}
\def\fC{{\mathfrak C}}

\def\eqref#1{(\ref{#1})}

\def\vec#1{\mathbf{#1}}

\def\squareforqed{\hbox{\rlap{$\sqcap$}$\sqcup$}}
\def\qed{\ifmmode\squareforqed\else{\unskip\nobreak\hfil
\penalty50\hskip1em\null\nobreak\hfil\squareforqed
\parfillskip=0pt\finalhyphendemerits=0\endgraf}\fi}

%%%%%%%%%%%%%%%%%%%%%%%%%
% Alphabet calligraphie %
%%%%%%%%%%%%%%%%%%%%%%%%%
\def\cA{{\mathcal A}}
\def\cB{{\mathcal B}}
\def\cC{{\mathcal C}}
\def\cD{{\mathcal D}}
\def\cE{{\mathcal E}}
\def\cF{{\mathcal F}}
\def\cG{{\mathcal G}}
\def\cH{{\mathcal H}}
\def\cI{{\mathcal I}}
\def\cJ{{\mathcal J}}
\def\cK{{\mathcal K}}
\def\cL{{\mathcal L}}
\def\cM{{\mathcal M}}
\def\cN{{\mathcal N}}
\def\cO{{\mathcal O}}
\def\cP{{\mathcal P}}
\def\cQ{{\mathcal Q}}
\def\cR{{\mathcal R}}
\def\cS{{\mathcal S}}
\def\cT{{\mathcal T}}
\def\cU{{\mathcal U}}
\def\cV{{\mathcal V}}
\def\cW{{\mathcal W}}
\def\cX{{\mathcal X}}
\def\cY{{\mathcal Y}}
\def\cZ{{\mathcal Z}}
\newcommand{\rmod}[1]{\: \mbox{mod} \: #1}

\def\vr{\mathbf r}

\def\e{{\mathbf{\,e}}}
\def\ep{{\mathbf{\,e}}_p}
\def\em{{\mathbf{\,e}}_m}

\def\Tr{{\mathrm{Tr}}}
\def\Nm{{\mathrm{Nm}}}

 \def\SS{{\mathbf{S}}}

\def\lcm{{\mathrm{lcm}}}

\def\({\left(}
\def\){\right)}
\def\fl#1{\left\lfloor#1\right\rfloor}
\def\rf#1{\left\lceil#1\right\rceil}

\def\eps{\varepsilon}
\def\al{\alpha}
\def\be{\beta}
\def\N{\mathbb{N}}
\def\L{\mathbb{L}}

\def\mand{\qquad \mbox{and} \qquad}
\def\mor{\qquad \mbox{or} \qquad}

\newcommand{\commA}[1]{\marginpar{%
\begin{color}{red}
\vskip-\baselineskip %raise the marginpar a bit
\raggedright\footnotesize
\itshape\hrule \smallskip A: #1\par\smallskip\hrule\end{color}}}

\newcommand{\commI}[1]{\marginpar{%
\begin{color}{magenta}
\vskip-\baselineskip %raise the marginpar a bit
\raggedright\footnotesize
\itshape\hrule \smallskip I: #1\par\smallskip\hrule\end{color}}}

\newcommand{\commM}[1]{\marginpar{%
\begin{color}{blue}
\vskip-\baselineskip %raise the marginpar a bit
\raggedright\footnotesize
\itshape\hrule \smallskip M: #1\par\smallskip\hrule\end{color}}}

%%%%%%%%%%%%%%%%%%%%%%%%%%%%%%%%%%%%%%%%%%%%%%%%%%%%%%%%
%%%%%%%%%%%%%%%%%%%%%%%%%%%%%%%%%%%%%%%%%%%%%%%%%%%%%%%%
%%%%%%%%%%%%%%%%%%%%%%%%%%%%%%%%%%%%%%%%%%%%%%%%%%%%%%%%
%%%%%%%%%%%%%%%%%%%%%%%%%%%%%%%%%%%%%%%%%%%%%%%%%%%%%%%%

%%%%%%%  END OF STANDARD STUFF %%%%%%%%%

%%%%%%%%%%%%%%%%%%%%%%%%%%%%%%%%%%%%%%%%%%%%%%%%%%%%%%%%
%%%%%%%%%%%%%%%%%%%%%%%%%%%%%%%%%%%%%%%%%%%%%%%%%%%%%%%%
%%%%%%%%%%%%%%%%%%%%%%%%%%%%%%%%%%%%%%%%%%%%%%%%%%%%%%%%
%%%%%%%%%%%%%%%%%%%%%%%%%%%%%%%%%%%%%%%%%%%%%%%%%%%%%%%
%%%%%%%%%%%
%%% Spell

\hyphenation{re-pub-lished}

\parskip 4pt plus 2pt minus 2pt

\mathsurround=1pt

\def\bfdefault{b}
\overfullrule=5pt

\def \F{{\mathbb F}}
\def \K{{\mathbb K}}
\def \Z{{\mathbb Z}}
\def \Q{{\mathbb Q}}
\def \R{{\mathbb R}}
\def \C{{\mathbb C}}
\def\Fp{\F_p}
\def \fp{\Fp^*}

\def\Kmn{\cK_p(m,n)}
\def\psmn{\psi_p(m,n)}
\def\AI{\cA_p(\cI)}
\def\BIJ{\cB_p(\cI,\cJ)}
\def \xbar{\overline x_p}

\title[Lattices from 
Algebraic Number Fields]{On Distances in Lattices from 
Algebraic Number Fields}

\author{Art\= uras Dubickas}
\address{Department of Mathematics and Informatics, Vilnius University, Naugarduko 24,
LT-03225 Vilnius, Lithuania}
\email{arturas.dubickas@mif.vu.lt}

\author{Min Sha}
\address{School of Mathematics and Statistics, University of New South Wales,
 Sydney, NSW 2052, Australia}
\email{shamin2010@gmail.com}

\author{Igor E. Shparlinski}
\address{School of Mathematics and Statistics, University of New South Wales,
 Sydney, NSW 2052, Australia}
\email{igor.shparlinski@unsw.edu.au}

\begin{abstract} 
In this paper, we study a classical construction of lattices from number fields
and obtain a series of new results about their minimum distance and 
other characteristics by introducing a new measure of algebraic numbers. 
In particular, we show that when the number fields have few complex embeddings, the minimum distances of these lattices can be computed exactly.\end{abstract}


\maketitle

\section{Introduction}

\subsection{Background}

 We first recall a classical and simple construction of lattices from number fields, which is usually applied to show that the class number of a number field is finite~\cite{Lang}. Let $\K$ be a number field
of degree $n$ over the rational numbers $\Q$ and of discriminant $D_{\K}$. We denote by 
$\Z_{\K}$
 the ring of integers of $\K$.  We say that  $\K$ is of {\it signature\/} $(s,t)$
if it has $s$ 
real embeddings, which are assumed to be $\sigma_i: \K \hookrightarrow \R$, $i =1, \ldots, s$ if $s \ge  1$,  and $t$  conjugate pairs of complex embeddings , which are denoted by $\tau_j, \overline\tau_j: \K \hookrightarrow \C$, $j=1, \ldots, t$ if $t \ge 1$. 
Thus, $s + 2t = n$.
Then, the embeddings $\sigma_1,\ldots,\sigma_s$ (if $s\ge 1$) and $\tau_1,\ldots,\tau_t$ (if $t \ge 1$)  naturally yield the vector embedding 
$${\psi \, \colon \, \K
\hookrightarrow \R^{s} \times \C^{t} \cong \R^{n}}, 
$$ 
and we 
can define the $n$-dimensional lattice $\Lambda _{\K} =\psi (\Z_{\K})$, 
see~\cite[Chapter~8, Section~7]{ConSlo}.

Lattices  $\Lambda _{\K}$ are very natural objects which play 
an important role in algebraic number theory and its applications,  
so they have been investigated 
from several different points of view, see, for 
instance,~\cite{B-F1,B-F2,B-F3,B-FOV,Gur,JdACS, Xing} and references therein. 
Besides, the interest to such lattices has also been 
stimulated by the work of Litsyn and Tsfasman~\cite{LiTs} which demonstrates 
that such lattices can be very useful in constructing dense ball packing of 
the Euclidean space  $\R^{s+2t}$; see also~\cite{RoTs,Tsfas} for further developments. 

To begin with, let us recall some simple facts about these lattices. 
For the determinant, it is well known that 
$$
\det \Lambda _{\K} = 2^{-t}|D_{\K}|^{1/2}
$$
(see~\cite[Chapter~V, Section~2, Lemma~2]{Lang}). 

The \textit{minimum distance} $d(\Lambda_{\K})$ of $\Lambda_{\K}$ is the minimum distance between any two distinct lattice points of $\Lambda_{\K}$  and in fact equals the length of the \textit{shortest non-zero vector}, that is,  
$$
d(\Lambda_{\K}) = \min_{\vec{x} \in \Lambda_{\K}\setminus \{\vec{0}\}}  \|\vec{x}\|_\K, 
$$  
where 
$$
 \|\vec{x}\|_\K= \(\sum^{s}_{i=1} x^{2}_{i} +
\sum^{t}_{j=1} (y^{2}_{j} + z^{2}_{j})\)^{1/2}
$$
for
$
\vec{x}= (x_{1}, \ldots\,, x_{s}, y_{1} + z_{1}\sqrt{-1}, \ldots\,,
y_{t} + z_{t}\sqrt{-1}) \in  
\Lambda_{\K}$. 
Notice that if $s=0$, then there are no such coordinates $x_1,\ldots,x_s$; while if $t=0$, then there are no such coordinates corresponding to the complex embeddings. We implicitly admit this fact without special indications throughout the paper. 

For brevity, we  write $\| \alpha\|_\K$ instead of 
$\|\psi(\alpha)\|_\K$ for $\alpha \in \K$;
hence, alternatively, 
$$
d(\Lambda_{\K}) = \min_{\alpha \in \Z_{\K}\setminus \{0\}}  \| \alpha\|_\K.
$$
For simplicity, throughout the paper we  also write 
$$
\| \al \| = \|\al\|_{\Q(\al)},
$$
and call this  quantity the {\it absolute  size\/} of $\alpha$. 
We also call $\|\al\|^2$ the {\it absolute  square size\/} of $\alpha$. 
Similarly, we can define $\|\al\|_{\K}$ and $\|\al\|_{\K}^2$ to be the \textit{relative size} and the \textit{relative square size} 
of $\al$ with respect to $\K$ respectively. 

%call the square of this  quantity, that is, $\| \al \|^2$,  the {\it absolute  size\/} of $\alpha$. 

Tsfasman~\cite[Lemma~1.1(ii)]{Tsfas} has shown that the minimum distance $d(\Lambda_{\K})$ of $\Lambda_{\K}$ satisfies
\begin{equation}
\label{eq:Bound T}
(s/2 + t)^{1/2}   \le d(\Lambda_{\K}) \le (s + t)^{1/2}.
\end{equation}
Here, the upper bound comes from the trivial example $\alpha=1 \in \K$.
The lower bound in~\eqref{eq:Bound T} has been improved in~\cite[Theorem~5.11]{Shp-Fq} 
as follows:
\begin{equation}
\label{eq:Bound S}
d(\Lambda _{\K}) \ge (s2^{-2t/(s+2t)} + t2^{s/(s+2t)})^{1/2}
= 2^{s/(2s+4t)}(s/2 + t)^{1/2}.
\end{equation}

\subsection{New point of view}

It is easy to see that  in the case  $st=0$ the bound in~\eqref{eq:Bound S} together with the upper bound in~\eqref{eq:Bound T}   implies 
$d(\Lambda_{\K}) = (s+t)^{1/2}$. 
This motivates us to define, 
for a number field $\K$ with signature $(s,t)$  and an algebraic number 
$\alpha \in \K$,  the following two normalised quantities: 
$$
m(\K) = \frac{d(\Lambda _{\K})^2 }{s+t} \mand m_{\K}(\alpha)=\frac{\|\alpha\|_{\K}^2}{s+t}.
$$
Clearly, 
\begin{equation}\label{apibr}
m(\K) = \min_{\alpha \in \Z_{\K}\setminus \{0\}} m_{\K}(\alpha). 
\end{equation}
If $\K= \Q(\alpha)$, we simply write 
$$
 m(\alpha) = m_{\Q(\alpha)}(\alpha) = \frac{\|\alpha\|^2}{s+t},
$$ 
and call $m(\alpha)$ the 
{\it absolute normalised square size\/} of $\alpha$, 
otherwise we call it the  {\it relative normalised square size\/} of $\alpha$ (with respect to the field $\K$).  
Similarly, we call $\sqrt{m(\alpha)}$ the 
{\it absolute normalised size\/} of $\alpha$ and call $\sqrt{m_{\K}(\alpha)}$
the  {\it relative normalised size\/} of $\alpha$ with respect to $\K$.

For the convenience of computations, we usually handle $m(\alpha)$ or $m_{\K}(\alpha)$ in our arguments and results. 
This automatically implies related results on $\sqrt{m(\alpha)}$ or $\sqrt{m_{\K}(\alpha)}$.

Below (see Theorem~\ref{thm: mK mQ}), we  show that for any number field $\K$ and any algebraic integer $\alpha \in \K$ we have
$$
m_{\K}(\alpha) \geq \min\{1,m(\alpha)\}, 
$$
when there exists $\beta \in \K$ such that $\K=\Q(\al,\be)$ and the fields $\Q(\al)$ and $\Q(\be)$ are linearly disjoint over $\Q$. 

Note that $m(\alpha)$ is at least $1$ for every {\it reciprocal} algebraic integer $\alpha$ (such that $\al^{-1}$ is its conjugate over $\Q$). Indeed, then $\|\al\|^2 \geq s+t$, since for every pair of conjugates $\al,\al^{-1}$ satisfying $|\al| \ne 1$ we have
$|\alpha|^2+|\alpha|^{-2} > 2$, whereas in the  case $\al$ is a complex number lying on $|z|=1$ the numbers $\al,\al^{-1}$ are complex conjugates, so we only take $|\al|^2=1$ into the sum of $s+t$ squares. Therefore,
unlike in the well-known Lehmer's problem about minimal Mahler measure,  the algebraic integers $\alpha$ satisfying $m(\alpha)<1$ must be non-reciprocal. 
By~\eqref{eq:mal large} below, the smallest value of $m(\alpha)$ is
greater than
$(e \log 2)/2 =0.942084 \ldots$, where $e$ is  the base of the natural logarithm. On the other hand, as we  show in Section~\ref{sec:num},
the smallest value among algebraic integers of degree at most $6$ is
$$m(\zeta) = 
\frac{\zeta^2+\zeta^{-1}}{2}=0.946467\ldots,$$ where 
$\zeta=0.826031\ldots$ is the root of $x^6+x^2-1=0$. 

By  the upper bound in~\eqref{eq:Bound T} and the lower  bound~\eqref{eq:Bound S}, we obtain
\begin{equation}
\label{eq:Bound mK}
\frac{s 2^{-2t/(s+2t)}+t2^{s/(s+2t)}}{s+t} \le  m(\K) \le 1.
\end{equation}
Thus, $m(\K)=1$ if   either $s=0$ or $t=0$, which automatically includes cyclotomic fields. 
In particular, it is interesting to investigate what happens when either $s$ or $t$ 
are small compared to $n$.
Here, we show that the cases of {\it small $s$\/} and of {\it small $t$\/} 
exhibit a  very {\it different\/} behaviour.

\subsection{Outline of the main results}

First, in Section~\ref{sec:small t} we use some other ideas to  show  that  if  $t$ is small compared to $s$ then $m(\K)=1$.  
In particular, by Theorem~\ref{thm:manyreal K}, for any sufficiently large $s$ and 
 for every number field $\K$ of signature $(s,t)$ with 
$$
 t \le   0.096  \sqrt{s/\log s}, 
$$
we have $m(\K)=1$. 

We note that the proof of~\eqref{eq:Bound S}  can be extended to 
the  lower bound 
\begin{equation}
\label{eq:d N} 
  \|\alpha\|_\K \geq \( s2^{-2t/(s+2t)} + t2^{s/(s+2t)}\)^{1/2}\left| \Nm_{\K}(\alpha)\right|^{1/(s+2t)},
\end{equation}
 where $\Nm_{\K}(\alpha)$ is the norm from $\K$ to $\Z$ of  $\alpha \in \K\setminus \{0\}$. 
We present this as Lemma~\ref{lem:d and N} (in a slightly different but equivalent 
form), which is also one of our tools. We also obtain  a lower bound on 
$\|\alpha\|_\K$ of a new type where instead of using the norm we use the 
discriminant of $\alpha$; see  Theorem~\ref{thm:manyreal alpha} below.

In Section~\ref{sec:large t}, we give  
 infinite series of fields $\K$  with $s$ and $t$ of comparable size 
for which $m(\K)<1$. Moreover, in Theorem~\ref{cubic1} we give
an infinite family of fields of signature $(s,s)$ for which 
$$0.944940\ldots  \le  m(\K) \le 0.947279\ldots.$$ In fact, the lower bound 
comes directly from~\eqref{eq:Bound mK}  which demonstrates
that it is quite precise for such signature as well. See also Theorem~\ref{cubic11} for a more precise upper bound when $s$ is even. 

Furthermore, in Section~\ref{sec:small s} 
using a classical result of Erd\H{o}s and Tur\'an~\cite{E--T}, 
we show that even for very small $s$ there are fields with  $m(\K)<1$.
More precisely, in Theorem~\ref{bhu1} 
we give an infinite series of examples of fields $\K=\Q(\al)$ of growing degree
with $s=1,2$ for which $m(\K) \leq m(\al)<1$. 
Moreover, select any non-zero algebraic integer $\al$ with exactly one real conjugate satisfying  $m(\al)<1$. Then, take $\K=\Q(\al,\be)$, where $\be$ is
a totally real algebraic number of degree $s$ over 
$\Q(\al)$ and also over $\Q$. In this way,
by Lemma~\ref{subfield} (see~\eqref{cfrt4}), we obtain $m(\K)<1$ for infinite series of fields $\K=\Q(\al,\be)$ of growing degree
with precisely $s$ real embeddings. 

The most interesting example we find corresponds 
to the case when $s \geq 2$ is even and $n$ is large compared to $s$. 
By~\eqref{eq:Bound mK}, defining 
$$\gamma = s/n
$$ 
we see that the inequality
$$
m(\K)  \geq  \frac{2^{s/n}}{1+s/n}=\frac{2^{\gamma}}{1+\gamma} = 1 -\gamma(1-\log 2) + O(\gamma^2)
$$
holds for any number field $\K$ with signature $(s,(n-s)/2)$, where $n$ is even.
On the other hand, by Theorem~\ref{kiy1} below, for each even $s \geq 2$ and each integer $k \ge 1$ there exists a field $\K$ of degree $n=(2k+1)s$ with signature $(s,(n-s)/2)$ satisfying 
$$
m(\K)  \le  1 -\gamma(1-\log 2) + O(\gamma^{5/4}), 
$$
where (here and elsewhere) the implied constants are absolute, unless stated otherwise.
 To construct this example we take  
the square root of a {\it multinacci number\/} of odd degree
and compute its absolute size.
(The positive root $\vartheta>1$ of the polynomial $x^n-x^{n-1}-\ldots-x-1$ is often called the multinacci number, see, for example,~\cite{sidorov}; although sometimes its reciprocal $\beta=\vartheta^{-1}$ 
is also called the same way~\cite{jordan}.) 
 
In Section~\ref{sec:num}, we give some numerical examples
of fields $\K$ with small degree satisfying  $m(\K)<1$. To conclude, we raise some questions related to this research.

\section{Preliminaries}

\subsection{Bounds of some products}

We often talk about {\it real\/} and {\it complex\/} conjugates of algebraic numbers. 
In this context  we always follow the convention that 
``complex'' means ``complex non-real''. We often apply the inequality of arithmetic and geometric means,  
which we abbreviate as AM-GM.

We now prove~\eqref{eq:d N} in an equivalent and sometimes slightly more convenient form.

\begin{lemma}
\label{lem:d and N}  Given a number field $\K$ of signature $(s,t)$, and an algebraic 
number $\alpha \in \K$, we have
\begin{equation}\label{derf0}
\|\alpha\|_{\K}^2 \geq n 2^{s/n-1}|\Nm_{\K}(\alpha)|^{2/n},
\end{equation}
where $n=s +2t$ is the degree of $\K$. Furthermore, if $\alpha \in \K$ is a non-zero algebraic integer, for $st>0$ we have 
\begin{equation}\label{derf1}
\|\alpha\|_{\K}^2 > n 2^{s/n-1}.
\end{equation}
\end{lemma}

\begin{proof}
We note that~\eqref{derf0} is equivalent to the inequality 
\begin{equation}\label{vienas11}
\|\alpha\|_{\K}^2 \ge n 2^{-2t/n}|\Nm_{\K}(\alpha)|^{2/n}.
\end{equation}

The norm of $\alpha$ in the field $\K$, $\Nm_{\K}(\alpha)$, can be written as $XY^2$, where $X$ is the product
of $\sigma(\alpha)$ over $s$ real embeddings $\sigma$ of $\K$, and
$Y^2$ is the corresponding product over $2t$ complex embeddings. 
For $st=0$, we obtain~\eqref{vienas11} immediately, 
by AM-GM. For instance, when $s=0$, we have 
$t=n/2$ and $|\Nm_{\K}(\alpha)|=Y^2$.  

Assume that $st>0$.
Then, by AM-GM, we obtain 
\begin{equation}\label{vienas14}
\|\alpha\|_{\K}^2 \geq s |X|^{2/s}+t|Y|^{2/t} = s|X|^{2/s}+t|\Nm_{\K}(\alpha)|^{1/t}|X|^{-1/t}.
\end{equation}
For $A>0$, the minimum of the function
$F(x)=sx^{2/s}+tA^{1/t} x^{-1/t}$  on the half line $x > 0$ is attained for 
$x_0 = 2^{-st/(s+2t)} A^{s/(s+2t)}$, 
which is the root of the equation $2x^{2/s} = A^{1/t} x^{-1/t}$,
and thus is equal to 
\begin{align*}
F(x_0) &=sx_0^{2/s}+tA^{1/t}x_0^{-1/t} =  (s+2t) x_0^{2/s} = (s+2t) 2^{-2t/(s+2t)} A^{2/(s+2t)} \\
&  = n2^{-2t/n} A^{2/n} = n2^{s/n-1} A^{2/n}. 
\end{align*}
Using this with $A = |\Nm_{\K}(\alpha)|$ we derive~\eqref{derf0}.

Now, we further assume that $\alpha$ is a non-zero algebraic integer. 
Note that $\Nm_{\K}(\alpha)$ is a non-zero integer. 
Clearly,~\eqref{derf0} implies~\eqref{derf1} in the case $|\Nm_{\K}(\alpha)|\ge 2$. Hence, 
from now on we assume not only that $st>0$ holds but also that $\alpha \in \K$ is a unit, so that
$A=|\Nm_{\K}(\alpha)|= |XY^2| =1$. Then, the 
equality in~\eqref{vienas14} holds if and only if for all $s$ real embeddings $\sigma_i$, $i=1,\ldots,s$, we have 
\begin{equation}\label{kiol1}
|\sigma_i(\alpha)|=|X|^{1/s}, 
\end{equation}
 and for $t$ complex embeddings $\tau_j$, $j=1,\ldots,t$, we have 
\begin{equation}\label{susm}
 |\tau_j(\alpha)|=|Y|^{1/t}=|X|^{-1/(2t)}
 \end{equation}
(then also $|\overline{\tau}_j(\alpha)|=|Y|^{1/t}$).
Furthermore, as above, we see that the function $F(x)$ attains its minimum (and so $\|\alpha\|_{\K}^{2}$ is equal to $n2^{s/n-1}$)
if, in addition to this, 
\begin{equation}\label{kiol3} 
 |X|=x_0=2^{-st/(s+2t)}.
 \end{equation}

In order to prove~\eqref{derf1} it remains to show that at least one of these equalities
 cannot hold.
Indeed, since $s>0$, the number $\alpha$ must have a real conjugate.   
By~\eqref{kiol1} and~\eqref{kiol3}, the modulus of this conjugate 
equals $|X|^{1/s}=2^{-t/(s+2t)}=2^{-t/n}<1$. So, one of the conjugates of $\alpha$ over $\Q$ must be equal to $\pm 2^{-t/n}$. However, such $\alpha$ has all conjugates (real and complex) on the same circle $|z|=2^{-t/n}$, so it is not
an algebraic integer, which leads to a contradiction.   

An alternative proof of the fact that no algebraic integer satisfies \eqref{kiol1}, \eqref{susm} and~\eqref{kiol3} can be given by employing the results of~\cite{sukr} about algebraic numbers whose all conjugates over $\Q$ lie on two circles. 
\end{proof} 

We now record the following useful statement. 

\begin{lemma}
\label{lem:nece}
If $\al$ is an algebraic integer  of degree  $n\le 23$ with
$m(\al)<1$, then $\al$ is an algebraic unit.
\end{lemma}

\begin{proof} 
Assume that $\al$ has $s$ real conjugates and $2t$ complex conjugates. Let $\K=\Q(\alpha)$, which is of signature $(s,t)$. 
Setting $y=s/n \in [0,1]$,  from
\[\frac{n2^{s/n-1}}{s+t}=\frac{2^{s/n}}{2s/n+2t/n}=\frac{2^{s/n}}{1+s/n}=\frac{2^y}{1+y}\]
and Lemma~\ref{lem:d and N}, we find that
\begin{equation} \label{vienas12a}
m(\alpha)=\frac{\|\alpha\|^2}{s+t} \geq 
\frac{2^y}{1+y} |\text{Nm}_{\K}(\alpha)|^{2/n}.
\end{equation}
The smallest value of the function $2^y/(1+y)$ in the interval $[0,1]$ 
occurs at the point $y_0=-1+1/\log 2= 0.442695 \ldots$. The minimum  is equal to $2^{y_0}/(1+y_0)=(e\log 2)/2=0.942084 \ldots$. Since $y_0$ is not a rational number, 
from~\eqref{vienas12a} we have  
\begin{equation}\label{vienas12}
m(\alpha) > \frac{e \log 2}{2}|\text{Nm}_{\K}(\alpha)|^{2/n} > 0.942084 |\text{Nm}_{\K}(\alpha)|^{2/n}.
\end{equation}
Observe that when $\alpha$ is not a
unit and $n \leq 23$, we have
$$0.942084 \cdot |\text{Nm}_{\K}(\alpha)|^{2/n} \geq 0.942084
\cdot  2^{2/23} >1.$$
So, using~\eqref{vienas12} and noticing the assumption $m(\al)<1$, we complete the proof. 
\end{proof} 

We see from  Lemma~\ref{lem:nece} that for number fields $\K$ of degree $n \leq 23$ in the search of the
minimum $m(\K)$ it suffices to consider only units $\alpha$ in $\K$.
We also remark that~\eqref{vienas12} implies that
for any algebraic integer  $\al$, we have 
\begin{equation}\label{eq:mal large}
m(\al) > \frac{e \log 2}{2} =0.942084 \ldots.
\end{equation} 
Moreover, applying the same arguments as in the proof of Lemma~\ref{lem:nece}, 
we can obtain 
\begin{equation}
m({\K}) > \frac{e \log 2}{2} =0.942084 \ldots
\end{equation}
for any number field $\K$.

Note that in the proof of Lemma~\ref{lem:nece}
the variable $y= s/n$ is  rational number and thus it never takes the optimal value $y_0$. In particular, for $n=24$
the critical value of $y$ is $y = 5/12$ (corresponding to fields
of the signature $(10,7)$). Unfortunately, this  is still not enough to
include the value of $n=24$ in Lemma~\ref{lem:nece}. However, this
observation allows us to slightly improve the bound~\eqref{eq:mal large}
for any finite range on $n$.

We need the following inequality due to Schur~\cite[Satz~II]{schur}. 

\begin{lemma}\label{lem:schur}
If $L,x_1,\ldots,x_s$ (where $s \geq 2$) are real numbers satisfying $x_1^2+\ldots+x_s^2 = L$, then
$$
\prod_{1 \leq i<j \leq s} (x_i-x_j)^2 \leq \(\frac{L}{s^2-s}\)^{(s^2-s)/2} \prod_{k=1}^s k^k.
$$
\end{lemma}

We also need a bound on the product of Lemma~\ref{lem:schur}
which is given by~\cite[Equation~(15)]{schur} (but here we briefly sketch 
a proof).

\begin{lemma}\label{lem:prod}
We have
$$
\prod_{k=1}^s k^k =  s^{(s^2+s)/2+1/12} e^{-s^2/4 +O(1)}.
$$
\end{lemma}

\begin{proof} First we write 
$$
\sum_{k=1}^s k\log k =  \sum_{k=2}^s k \(\log s+ \log (k/s)\)
=  \frac{(s-1)(s+2)}{2} \log s + s^2 \sigma  ,
$$
where 
$$
\sigma =\frac{1}{s}  \sum_{k=2}^s\frac{k}{s}   \log \frac{k}{s}.
$$
Now, using the Euler-Maclaurin summation formula one easily 
derives
that 
$$
\sigma =\int_{1/s}^1 x \log x \> dx + \frac{7}{12}s^{-2}\log s + O(s^{-2}) = -\frac{1}{4} + \frac{13}{12}s^{-2}\log s + O(s^{-2}),
$$
which concludes the proof. 
\end{proof}

\subsection{Distribution of roots of some polynomials}

The next result is due to Erd\H{o}s and Tur\'an~\cite[Theorem~I]{E--T}.

\begin{lemma}\label{ertur}
Let $N_P(\phi, \varphi)$ be the number of roots of 
a complex polynomial
$P(x)=a_d x^d+ \ldots + a_0$, where $a_d a_0 \ne 0$, 
whose arguments belong to the interval 
$[\phi, \varphi) \subseteq [0,2\pi)$. Then,
$$
\left|N_P(\phi, \varphi)-\frac{\varphi-\phi}{2\pi} d\right|
 \leq 16 \sqrt{d\log \(L(P)/\sqrt{|a_d a_0|}\)},
$$
where $L(P)=|a_0|+ \ldots + |a_d|$. 
\end{lemma}

Ganelius~\cite{Gan} had replaced the constant $16$ by the smaller constant $\sqrt{2\pi/G}=2.619089\ldots$, where $G=\sum_{j=1}^{\infty}(-1)^{j+1}/(2j-1)^2$ is {\it Catalan's constant\/}. 

Recall that a Pisot number is a real algebraic integer greater than 1 and all its conjugates lie inside the open unit disk. 
By definition, if a monic integer polynomial defines a Pisot number, then it is necessarily irreducible. 

It has been proved that the polynomial $x^n-x^{n-1}-\ldots -1$ defines a Pisot number, which is located in the interval $(1,2)$; see~\cite{Miles}. 
Here, we need to know more precisely about the locations of all its roots.

\begin{lemma}\label{bhu0}
Fix $n \geq 2$. Then, the  irreducible polynomial
$$f(x)=x^n-x^{n-1}-x^{n-2}- \ldots-1$$
defines a Pisot number $\vartheta$ that lies in the interval
\begin{equation}\label{didz}
2n/(n+1)<\vartheta<2.
\end{equation}
It has another real conjugate $\omega$ 
if and only if $n$ is even, and this second real conjugate lies in the interval 
\begin{equation}\label{didz1}
-1<\omega<-3^{-1/n}. 
\end{equation}
Moreover, the other conjugates of $\vartheta$ are non-real and all lie in the annulus
\begin{equation}\label{didz2}
3^{-1/n}<|z|<1.\end{equation}
\end{lemma}

\begin{proof}
 Set
$$g(x)=f(x)(x-1)=x^{n+1}-2x^n+1.$$
Note that the derivative $g'(x)=((n+1)x-2n)x^{n-1}$ is zero at $x=0$ and at 
$x_0=2n/(n+1) \in (1,2)$. So, $g$ is decreasing in $(-\infty,x_0)$,
and increasing in $(x_0,+\infty)$ for $n$ odd. Similarly, for  $n$ even, 
$g$ is increasing in $(-\infty,0)$, decreasing in $(0,x_0)$, and then increasing in $(x_0,+\infty)$.  
Besides, $g(1)=0$ and $g(2)=1>0$. So, $g$ has a unique real root, say $\vartheta$, in the interval $(x_0,2)$, which gives~\eqref{didz}. Also, $g$ has no real roots in $[2,+\infty)$. Furthermore, we can see that for any real $a$ with $1<a<\vartheta$, we have $g(a)<0$, that is $a^{n+1}+1< 2a^n$. Thus, by {\it Rouch\'e's theorem\/}, 
$g$ has exactly $n$ roots in the disc $|z|<a$, and so $f$ has exactly $n-1$ roots in the disc $|z|<a$. 
Hence, all the other roots of $f$ lie in the unit disc $|z| \le1$.
It is clear that $f$ has no roots on the circle $|z| =1$,  so $\vartheta$ is a Pisot number.
 
In addition, $g(-1)$ is negative for $n$ even 
and positive for $n$ odd. Thus, in view of $g(0)>0$, the 
polynomial $g$ (and so $f$) has a unique negative root, say $\omega$, in  $(-1,0)$ for $n$ even, and no negative 
roots for $n$
odd. Now, by looking at the values of $g$ at the endpoints of the interval
$[-1,-3^{-1/n}]$, one easily gets~\eqref{didz1}.  

Let  $\vartheta_1, \ldots,\vartheta_{n-1}$ be all roots of $f$ in the disk $|z| < 1$. 
Then, from $g(\vartheta_j)=\vartheta_j^{n+1}-2\vartheta_j^n+1=0$ and $|\vartheta_j|<1$ for each $j=1, \ldots,n-1$, it follows that $|\vartheta_j^n|=1/|2-\vartheta_j|>1/3$.
This implies~\eqref{didz2} and concludes the proof. 
\end{proof}

We also need to prove the irreducibility of certain polynomials. 

\begin{lemma}\label{auxil2}
The polynomial $x^{n}+x^{n-2}+x^{n-4}+\ldots+x^2-1$, where $n=4k+2 \in \Z$, $k \ge 1$, is irreducible over $\Q$ and has precisely two real roots. 
\end{lemma}

\begin{proof} Set $f(x)=x^{n/2}+x^{n/2-1}+\ldots+x-1$, then $f(x^2)$ is exactly the polynomial in the lemma. Since 
$n/2=2k+1$ is odd, and $-f$ is reciprocal to the polynomial considered in Lemma~\ref{bhu0}, its root $\be_{n/2}$ is real
and positive, and other $n/2-1$ roots are complex. 
By Lemma~\ref{bhu0}, we know that $\be_{n/2} > 1/2$. 
Note that $f$ is irreducible, then it is easy to see that
$f(x^2)$ is either irreducible with precisely two real roots $\pm \sqrt{\beta_{n/2}}$ or $f(x^2)=(-1)^{n/2}h(x)h(-x)=-h(x)h(-x)$ for some irreducible polynomial $h\in \Z[x]$.   

To show that the latter case is impossible we change the polynomials into their reciprocals. Then, since the constant coefficient of $h$ is $\pm 1$, we must have
$$-x^{n}f(1/x^2)= x^{n}-x^{n-2}-\ldots-x^2-1=-g(x)g(-x)$$
for some monic irreducible polynomial $g$ with integer coefficients, 
which has a root $1/\sqrt{\beta_{n/2}}$.  
Since $-1/\sqrt{\beta_{n/2}}$ is the root of $g(-x)$ and, by Lemma~\ref{bhu0}, these two roots of $x^{n}f(1/x^2)$ are its only roots outside the unit circle, the number $1/\sqrt{\beta_{n/2}}<\sqrt{2}=1.414213\ldots$ must be a Pisot number. Thus, $g$ is the minimal polynomial of a Pisot number.

However, by the results of Siegel~\cite{siegel}, Dufresnoy and Pisot 
\cite{duf0, duf} (see also~\cite{bertin}), there are only two Pisot numbers smaller than $1.42$: one with minimal polynomial $x^3-x-1$ and another with minimal polynomial $x^4-x^3-1$. Since $\deg g=n/2=2k+1$ is odd, it remains to check the only possibility $g(x)=x^3-x+1$ when $n=6$. Then, we see that
$$-g(x)g(-x)=(x^3-x+1)(x^3-x-1)=x^6-2x^4+x^2-1$$
is not equal to $x^{6}-x^4-x^2-1$. 
This in fact completes the proof. 
\end{proof}

We conclude this subsection with the next lemma
which is crucial in the proofs of 
Theorems~\ref{bhu1}  and~\ref{kiy}.

\begin{lemma}\label{lopo}
Let $\beta_1,\beta_2, \ldots,\beta_t$ be all the roots of the polynomial $x^n+x^{n-1}+\ldots+x-1$ with positive imaginary parts, where $n\ge 3$.
Then, for each $q>0$ we have
\begin{equation}\label{beeta}
\sum_{j=1}^t |\beta_j|^{q} =t+\frac{q}{2} \log 2+O(n^{-1/4}),
\end{equation}
where the constant in $O$ depends only on $q$.
\end{lemma}
  
\begin{proof}   
We first observe that the polynomial $-(x^n+x^{n-1}+\ldots+x-1)$ 
is the reciprocal polynomial of $f$ defined in Lemma~\ref{bhu0}. 
So, we indeed have $t\ge 1$ since $n \ge 3$. 
We set $\beta_j=\vartheta_j^{-1}$, where $\vartheta_j$, $j=1, \ldots,n$, are the roots of $f$ defined in Lemma~\ref{bhu0} (and in its proof).
 Note that~\eqref{didz1} and~\eqref{didz2} implies that 
\begin{equation}\label{nhyt2}
1<|\beta_j| < 3^{1/n}
\end{equation}
for each $j \leq n-1$. Here, $t=(n-1)/2$ for $n$ odd and $t=(n-2)/2$ for $n$ even, and the conjugates are labelled so that the imaginary parts of $\be_1,\ldots,\be_t$ are positive.

In order to evaluate the sum $\cC_{q}=\sum_{j=1}^t |\beta_j|^q$ we
consider the following $k$ sectors with arguments in
the intervals $[\pi j/k, \pi(j+1)/k)$ for $j=0,1, \ldots,k-1$:
$$
S_j=\{z \in \C \>:\>  1<|z|<3^{1/n},   \quad \pi j/k \leq \arg z< \pi(j+1)/k\}.
$$
Taking into account~\eqref{nhyt2} and the definition of $S_j$, we see that the roots
$\beta_1, \ldots,\beta_t$ all lie in the union of the above
$k$ sectors $\bigcup_{j=0}^{k-1} S_j$. 

Applying Lemma~\ref{ertur} to the polynomial 
$$
(x^n+x^{n-1}+\ldots+x-1)(x-1)=x^{n+1}-2x+1
$$
of degree $d=n+1$, we find that the number $N_j$ of its
roots $\beta_1,\ldots,\beta_t$ lying in $S_j$ satisfies
$$
|N_j-(n+1)/(2k)| \le 16\sqrt{(n+1)\log 4}+1,
$$
where the extra term $1$ reflects a possible real root $1$
(for $j=0$). 
Hence, selecting, for instance, 
\begin{equation}\label{kkkk}
k=\lfloor n^{1/4} \rfloor
\end{equation}
 and using
$n-2t \in \{1,2\}$, we find that
\begin{equation}\label{etrur1}
N_j=t/k+O(\sqrt{n}).
\end{equation}

For the diameter $\delta$ of the sector $S_j$,  
we clearly have 
\[
\delta  \leq 2 \cdot 3^{1/n} \sin(\pi/(2k))+3^{1/n}-1 < 
\pi 3^{1/n}/k+3^{1/n}-1 =O(1/k).
\]
For any root $\gamma$ of $x^{n+1}-2x+1$ lying in $S_j$, we have $| \gamma|^{n+1}=|2\gamma - 1|$. 
Since the distance from $ \gamma$ to $e^{\pi j \sqrt{-1}/k}$ is less than $\delta = O\(1/k\)$, we obtain
$$
|2 \gamma-1|=|2e^{\pi j \sqrt{-1}/k}-1|+O(1/k) =\sqrt{5-4\cos(\pi j/k)}+O(1/k).
$$
Hence,
\begin{align*}
| \gamma|^q &=|2 \gamma-1|^{q/(n+1)}=\(5-4\cos(\pi j/k)+O(1/k)\)^{q/(2n+2)}\\
& =1+\frac{q\log (5-4\cos(\pi j/k))}{2n+2}+O\(\frac{1}{kn}\) \\
&
=1+\frac{q\log (5-4\cos(\pi j/k))}{4t}+O\(\frac{1}{kn}\).
\end{align*}
Combining this with~\eqref{kkkk} and~\eqref{etrur1}, we see that
the total contribution of squares of the $N_j$ roots in the $j$-th sector $S_j$
into the sum $\cC_q$ is 
$$R_j = N_j+ \frac{q\log (5-4\cos(\pi j/k))}{4k}+O(n^{-1/2}).$$

Since $\sum_{j=0}^{k-1} N_j=t$, summing over $j$, $0 \leq j \leq k-1$, we deduce that
\[
\cC_q =\sum_{j=0}^{k-1} R_j = t+ \frac{q}{4}\sum_{j=0}^{k-1} \frac{\log (5-4\cos(\pi j/k))}{k} +O(kn^{-1/2}).\]
Replacing the sum $k^{-1}\sum_{j=0}^{k-1}\log (5-4\cos(\pi j/k))$
by the corresponding integral $\int_{0}^{1} \log(5-4\cos(\pi x))\> dx$ introduces the error of size  $O(1/k)$. Hence, 
recalling the choice of $k$ in~\eqref{kkkk},  we find that 
$$\cC_q= t+\frac{q}{4}\int_{0}^{1} \log(5-4\cos(\pi x))\> dx+O(n^{-1/4}).$$

Now, from the standard formula
\begin{align*}
\int_{0}^{1} \log(a+b\cos(\pi x))\> dx &=\frac{1}{\pi} \int_{0}^{\pi} \log(a+b\cos x) \> dx \\
&=\log\(\frac{a+\sqrt{a^2-b^2}}{2}\) 
\end{align*}
for $a=5$ and $b=-4$, it follows that the integral involved in $\cC_q$ is equal to $\log 4=2\log2$. 
Consequently,
\[\cC_q=t+\frac{q}{2}\log 2+O(n^{-1/4}),\]
which completes the proof of the lemma.
\end{proof}

\subsection{Absolute and relative normalised sizes}

Before going further, in this subsection we discuss the relations between absolute and relative normalised sizes. This could be of independent interest. 

Given an algebraic number  $\alpha$ of degree 
$n=s+2t$ with $s$ real and  $2t$ complex conjugates 
$$
\al_1,\ldots,\al_s, \al_{s+1},\ldots,\al_{s+t},\overline{\al_{s+1}}, \ldots, \overline{\al_{s+t}},
$$ 
it is convenient to define
\begin{equation}
\label{bhyt}
\cR(\al) =\sum_{i=1}^{s} \al_i^2 \mand \cC(\al) =\sum_{i=1}^{t} |\al_{s+i}|^2.
\end{equation}
By convention, we set $\cR(\al)=0$ if $s=0$, and $\cC(\al)=0$ if $t=0$. With this notation we have $$\|\al\|^2=\cR(\al)+\cC(\al)$$ and 
$$m(\al)= \frac{\|\al\|^2}{s+t}=\frac{\cR(\al)+\cC(\al)}{s+t}.$$

 In particular, we often use the fact that, by AM-GM, 
\begin{equation}
\label{eq:RC AMGM}
\cR(\al) \ge s\(\prod_{i=1}^{s} |\al_i|\)^{2/s} \mand  \cC(\al) \ge t\(\prod_{i=1}^{t} |\al_{s+i}|\)^{2/t}.
\end{equation}

\begin{lemma}\label{subfield} 
Let  $\al$ be an arbitrary algebraic number having $s_1$ real and $2t_1$ complex
conjugates over $\Q$, and let
$\be$ be an algebraic number having $s_2$ real and $2t_2$ complex conjugates over $\Q$ and degree $n_2=s_2+2t_2$ over the field $\Q(\alpha)$. 
 Then, the signature $(s,t)$ of $\K=\Q(\alpha,\beta)$ is 
\begin{equation}\label{cfrt1}
(s_1s_2, s_1t_2+s_2t_1+2t_1t_2),
\end{equation}
 and
\begin{equation}\label{eq:alpha RC}
\|\al\|_{\K}^2=(s_2+t_2)\cR(\al)+(s_2+2t_2) \cC(\al).
\end{equation}
Furthermore, if $t_2=0$ and $\al$ is an algebraic integer, we have 
\begin{equation}\label{cfrt3}
\|\al\|_{\K}^2= n_2 \|\al\|^2,
\end{equation}
and 
\begin{equation}\label{cfrt4}
m(\K) \leq m(\al).
\end{equation}
\end{lemma}

\begin{proof}
The number $\al$ is of degree $n_1=s_1+2t_1$ over $\Q$ with $s_1$ real conjugates $\al_j$, $1\leq j \leq s_1$, and $2t_1$ complex conjugates $\al_j$, $s_1+1 \leq j \leq n_1$ such that $\al_{s_1+t_1+i}=\overline{\al_{s_1+i}}$ for $1\le i \le t_1$.  Similarly, $\be$ has $s_2$ real conjugates $\be_j$, $1 \leq j \leq s_2$, and $2t_2$ complex conjugates $\be_j$, $s_2+1 \leq j \leq n_2$, over $\Q$ such that $\be_{s_2+t_2+i}=\overline{\be_{s_2+i}}$ for $1\le i \le t_2$. Furthermore, since $\K=\Q(\al,\be)=\Q(\al+r\be)$ for some $r \in \Q$, the primitive element $\al+r\be$ of the field $\K$ has $n=n_1n_2$ distinct conjugates $\al_i+r\be_j$, where $1 \leq i \leq n_1$ and $1 \leq j \leq n_2$, over $\Q$. Hence, its conjugate $\al_i+r\be_j$ is real if and only if $1 \leq i \leq s_1$ and $1 \leq j \leq s_2$. Indeed, it cannot be real if at least one of the numbers $\al_i, \be_j$ is complex, since otherwise it is equal to its complex conjugate $\overline{\al_i}+r\overline{\be_j}$ which should be distinct from $\al_i+r\be_j$ when either $i>s_1$ or $j>s_2$. Consequently, $\K$ has exactly
$s=s_1s_2$ real embeddings. This implies~\eqref{cfrt1}, because $$n-s_1s_2=n_1n_2-s_1s_2=2s_1t_2+2s_2t_1+4t_1t_2.$$

Next, observe that in the sum $\|\al\|_{\K}^2$ containing $s_1s_2+s_1t_2+s_2t_1+2t_1t_2$ squares of conjugates of $\al$, each real conjugate $\al_i$, $1 \leq i \leq s_1$, appears $s_2$ times
under real embedding and $t_2$ times under complex embedding of $\K$, whereas each complex conjugate $\al_i$, where $s_1+1 \leq i \leq s_1+t_1$, appears $n_2=s_2+2t_2$ times under complex embedding of $\K$.  Hence, taking into account~\eqref{bhyt} we 
obtain~\eqref{eq:alpha RC}.  

It is evident that for $t_2=0$ we have $n_2=s_2$, so $\|\al\|^2=\cR(\al)+\cC(\al)$ combined with~\eqref{eq:alpha RC} implies~\eqref{cfrt3}.

Finally, selecting this
$\alpha$ in~\eqref{apibr} and using~\eqref{cfrt1}, \eqref{cfrt3} and $t_2=0$, we derive that
\begin{align*} m(\K) &\leq m_{\K}(\al)= \frac{\|\al\|_{\K}^2}{s_1s_2+s_1t_2+s_2t_1+2t_1t_2}\\&
= \frac{n_2\|\al\|^2}{s_2(s_1+t_1)}=\frac{\|\al\|^2}{s_1+t_1}=m(\al),\end{align*}
which proves~\eqref{cfrt4}. 
\end{proof}

It is worth pointing out that in Lemma~\ref{subfield} we actually assume that the fields $\Q(\al)$ and $\Q(\be)$ are linearly disjoint over $\Q$, which means that  $[\Q(\al,\be):\Q]=[\Q(\al):\Q][\Q(\be):\Q]$. 
Let $\al$ be an algebraic number in a number field $\K$. If there is no $\be \in \K$ such that $K=\Q(\al,\be)$ and the fields $\Q(\al)$ and $ \Q(\be)$ are linearly disjoint over $\Q$, then~\eqref{eq:alpha RC} may be not true. For example, we can take $\K=\Q(2^{1/4})$ and $\al=\sqrt{2}$. Then, $\K$ is of degree $2$ over $\Q(\al)$, $\|\al\|^2=\cR(\al)=4$, and $\|\al\|_{\K}^2=6$; since in this case $(s_2,t_2)=(2,0)$ or $(0,1)$ for any quadratic $\beta$, the equation~\eqref{eq:alpha RC} does not hold. 

The following result is a tool to construct number fields $\K$ with $m(\K)<1$. 

\begin{theorem}
\label{thm: mK<1}
Let  $\al$ be an arbitrary algebraic integer having $s_1$ real and $2t_1$ complex
conjugates over $\Q$ such that $m(\alpha) < 1$.
Let $\be$ be an algebraic number having $s_2$ real and $2t_2$ complex conjugates over $\Q$ of degree $n_2=s_2+2t_2$ over $\Q(\alpha)$. Put $\K=\Q(\al,\be)$. Then, $m_{\K}(\alpha)< 1$
 if and only if
$$
\frac{t_2}{s_2+t_2} < \frac{s_1+t_1-\cR(\al) -\cC(\al) }{\cC(\al) -t_1}.
$$
\end{theorem}

\begin{proof} 
Let 
$\al_1,\ldots,\al_{s_1}$ be real
and    $\al_{s_1+1},\ldots,\al_{s_1+t_1},\overline{\al_{s_1+1}}, \ldots, \overline{\al_{s_1+t_1}}$ be 
$2t_1$ complex conjugates of $\alpha$. 
First, since $\al$ is an algebraic integer, we have
\begin{equation}\label{mkiuj}
|\al_1\cdots \al_{s_1}|\cdot |\al_{s_1+1}\cdots \al_{s_1+t_1}|^2 \ge 1.
\end{equation}
In the case $|\al_1\cdots \al_{s_1}| \ge 1$, this gives
$$
\(|\al_1\cdots \al_{s_1}|\cdot |\al_{s_1+1}\cdots \al_{s_1+t_1}|\)^2 \ge 1,
$$
which, together with~\eqref{bhyt} and AM-GM implies
$$
\cR(\al) + \cC(\al)= \sum_{i=1}^{s_1} \al_i^2 + \sum_{i=1}^{t_1} |\al_{s_1+i}|^2 \ge s_1 + t_1.
$$
This contradicts to the assumption $$m(\al)=\frac{\cR(\al) + \cC(\al)}{s_1+t_1}<1.$$

So, under the assumption $\cR(\al) + \cC(\al) < s_1+t_1$, we must have $|\al_1\cdots \al_{s_1}| < 1$,
and thus, by~\eqref{mkiuj}, $|\al_{s_1+1}\cdots \al_{s_1+t_1}| > 1$. Then, from~\eqref{eq:RC AMGM} it follows that
$$
 \cC(\al) \ge t_1 |\al_{s_1+1}\cdots \al_{s_1+t_1}|^{2/t_1} > t_1.
$$
Consequently, as $\cR(\al) + \cC(\al) < s_1+t_1$, we must have $\cR(\al) < s_1$.

Next, using~\eqref{cfrt1} and~\eqref{eq:alpha RC} it is easy to see that the inequality $\|\al\|_{\K}^2<s_1s_2+s_1t_2+s_2t_1+2t_1t_2$ is equivalent to 
$$
\frac{t_2}{s_2+t_2} < \frac{s_1+t_1-\cR(\al) -\cC(\al) }{\cC(\al) -t_1}, 
$$
which concludes the proof.
\end{proof}

Theorem~\ref{thm: mK<1} tells us that starting from an algebraic integer $\al$ with $m(\al)<1$, we can construct infinitely many number fields $\K$ with $m(\K) < 1$. 

We now estimate $m_{\K}(\alpha)$ in terms of $m(\alpha)$ in a special case.

\begin{theorem}
\label{thm: mK mQ}
Given a number field $\K$ and an algebraic integer $\alpha \in \K$. If there exists $\beta \in \K$ such that $\K=\Q(\al,\be)$ and the fields $\Q(\al), \Q(\be)$ are linearly disjoint over $\Q$, then we have
$$m_{\K}(\alpha) \geq \min\{1,m(\alpha)\}.$$
\end{theorem}

\begin{proof}
There is nothing to prove if $m_{\K}(\alpha) \geq 1$
or if $\K=\Q(\alpha)$. Thus, in the following, assume that
$m_{\K}(\alpha)<1$ 
and that $\K$ is a proper extension of $\Q(\alpha)$.
We need to prove the inequality
$$
m_{\K}(\alpha) \geq m(\alpha).
$$
Under our assumption,
$\beta \in \K=\Q(\alpha,\beta)$ has degree $[\K:\Q(\alpha)]$
over $\Q$. 

Below, we use the notations as those in the proof of Lemma~\ref{subfield}.
Then, recalling~\eqref{cfrt1} and~\eqref{eq:alpha RC}, we need to establish the inequality
\begin{equation}
\label{eq: target1}
m_{\K}(\al)=\frac{(s_2+t_2)\cR(\al) +(s_2+2t_2) \cC(\al)}{s_1s_2 + s_1t_2+s_2t_1+2t_1t_2}
\ge \frac{\cR(\al) +\cC(\al)}{s_1 + t_1}.
\end{equation}
After clearing the denominators, one can see that~\eqref{eq: target1} is equivalent to 
\begin{equation}
\label{eq: target2}
 s_1t_2 \cC(\al) \ge t_1t_2  \cR(\al).
\end{equation}

Put $\rho=(s_2+2t_2)/(s_2+t_2)$. Then, our assumption  $m_\K(\alpha)< 1$ is equivalent to the inequality 
\begin{equation}
\label{eq: ineq1}
\cR(\al) - s_1 < \rho (t_1 - \cC(\al)).
\end{equation}
On the other hand, from AM-GM and noticing that $\al$ is an algebraic integer, we have 
\begin{equation*}
\cR(\al) +  2\cC(\al)  \geq (s_1+2t_1)\left| \Nm_{\K}(\alpha)\right|^{1/(s_1 +2t_1)} \ge s_1+2t_1,
\end{equation*}
which implies that 
\begin{equation}
\label{eq: ineq2}
\cR(\al) - s_1 \ge   2(t_1-\cC(\al)). 
\end{equation}

From~\eqref{eq: ineq1} and~\eqref{eq: ineq2} it follows that $(\rho - 2)(t_1 - \cC(\al))>0$. Since $1 \le \rho \le 2$, we must have $ \cC(\al) >t_1$, and 
then $\cR(\al) < s_1$, by~\eqref{eq: ineq1}. This implies the inequality~\eqref{eq: target2}, and so the desired result. 
\end{proof}

Under the assumption in Theorem~\ref{thm: mK mQ} and from~\eqref{eq: target2} in the above proof, we can see that $m_\K(\al) > m(\al)$ if and only if $t_2>0$ and $s_1\cC(\al) > t_1\cR(\al)$. Thus, each of the following three cases can possibly happen: $m_\K(\al) > m(\al), m_\K(\al) = m(\al)$, or  $m_\K(\al) < m(\al)$. If we assume that $t_2>0$, then these three cases correspond to $s_1\cC(\al) > t_1\cR(\al), s_1\cC(\al) = t_1\cR(\al)$, and  $s_1\cC(\al) < t_1\cR(\al)$, respectively. For example, to see that all these cases can actually occur one can take the roots $\alpha$ of the polynomials $x^3+x+1, x^3-2$, and $x^3-x+1$, respectively. 
However, 
if $m_{\K}(\alpha)<1$ then as it has been just shown in the proof of Theorem~\ref{thm: mK mQ} we must have  $m_\K(\al) \ge m(\al)$.

Given a non-zero algebraic integer $\al$, which is not an algebraic unit, and an  integer $n$, it is clear that $m(\al^n)$ can become very large when $n$ is large. The following result says that $m(\al^{1/n})$ cannot be smaller than $1$ for large $n$, where $\al^{1/n}$ is an arbitrary $n$-th root of $\al$. However, if $\al$ is an algebraic unit, the situation might be different; see Theorem~\ref{cubic2} below. 

\begin{theorem}
\label{thm:al 1/n} 
Given a non-zero algebraic integer $\al$ with norm $\Nm(\al)$,  
let $n$ be an odd integer such that 
$\Q(\al^{1/n})$ has signature $(s,t)$ and satisfies $[\Q(\al^{1/n}):\Q(\al)]=n$. Then,  
$$
m(\al^{1/n})=1+\frac{\log |\Nm(\al)|}{s+t} +O\(\frac{1}{n^2}\).
$$
\end{theorem}

\begin{proof}
Assume that $\al$ has $s_1$ real conjugates $\al_1,\ldots,\al_{s_1}$ 
and  $2t_1$ complex conjugates  $\al_{s_1+1},\ldots,\al_{s_1+t_1},\overline{\al_{s_1+1}}, \ldots, \overline{\al_{s_1+t_1}}$. 
Let $\K=\Q(\al^{1/n})$. 
Note that, since $n$ is odd and $[\K:\Q(\al)]=n$, the signature $(s,t)$ of $\K$ satisfies $s=s_1$ and 
$$
t= \frac{n(s_1+2t_1)-s_1}{2} = \frac{1}{2}(n-1)s_1 + nt_1. 
$$
Hence, 
\begin{equation}
\label{eq:s+t}
s + t = \frac{1}{2}(n+1)s_1 + nt_1. 
\end{equation}

Since $[\K:\Q(\al)]=n$, as in the proof of~\eqref{eq:alpha RC}, we obtain 
$$
\cR(\al^{1/n})=\sum_{i=1}^{s_1}|\al_i|^{2/n},
$$
and 
$$
\cC(\al^{1/n}) = \frac{n-1}{2}\sum_{i=1}^{s_1}|\al_i|^{2/n} + 
n \sum_{j=1}^{t_1} |\al_{s_1+j}|^{2/n}.$$
Hence, 
\begin{equation}\label{mklmkl}
\|\al^{1/n}\|^2=\cR(\al^{1/n}) + \cC(\al^{1/n})=\frac{n+1}{2}
\sum_{i=1}^{s_1} |\al_i|^{2/n}+n \sum_{j=1}^{t_1} |\al_{s_1+j}|^{2/n}.
\end{equation}

Recalling the Taylor expansion of the exponential function for $i=1,\ldots,s_1$ and $j=1,\ldots,t_1$, we get 
$$
|\al_i|^{2/n} = \exp \(\frac{2}{n}\log |\al_i|\) = 1 + \frac{2\log |\al_i|}{n} + O\(\frac{1}{n^2}\)
$$
and, similarly, 
$$
|\al_{s_1+j}|^{2/n} =  1 + \frac{2\log |\al_{s_1+j}|}{n} + O\(\frac{1}{n^2}\).
$$
Thus,  by~\eqref{mklmkl},
$$
(s+t)m(\al^{1/n})=\|\al^{1/n}\|^2= \frac{1}{2}(n+1)s_1 + nt_1 + \log |\Nm(\al)| + O\(\frac{1}{n}\). 
$$
This finishes the proof,  by~\eqref{eq:s+t}. 
\end{proof}

In particular, 
we see from Theorem~\ref {thm:al 1/n} that $m(\al^{1/n})>1$ provided that  $|\Nm(\al)| \ge 2$ and $n$ is large enough.

\section{Main results}

\subsection{Fields with few complex embeddings}
\label{sec:small t}

We recall that if $t =0$ then $m(\K)=1$. Here we show that 
the same holds for a much wider class of number fields, namely, for fields with 
few complex embeddings. 

For an algebraic number  $\alpha\ne 0$ we denote by  $\Delta(\alpha)$ the discriminant of $\alpha$, which is the discriminant of its minimal polynomial over $\Z$. 

\begin{theorem}
\label{thm:manyreal alpha}
 For real $\eta > \kappa >0$ and $c \geq 1$, there is a positive integer $s(c,\eta,\kappa)$ such that
for every algebraic integer  $\alpha$ of degree 
$d$ over $\Q$ we have
\begin{equation*}
\max\left\{\cR_0(\al),\, \gamma/c \right\}
\geq  de^{2\kappa} |\Delta(\alpha)|^{2/(d^2-d)},
\end{equation*}
where $\cR_0(\al)$ is the sum of squares of any $s_0 \ge 2$ real conjugates of $\al$, $\gamma$ is the square of the maximal modulus of the remaining $d-s_0$ conjugates of $\al$, $s_0\ge s(c,\eta,\kappa)$ and 
$ d-s_0\le(1/4-\eta) s_0/\log s_0$. 
\end{theorem}

\begin{proof}  
Let 
$\al_1,\ldots,\al_s$ be $s$ real
and    $\al_{s+1},\ldots,\al_{s+t},\overline{\al_{s+1}}, \ldots, \overline{\al_{s+t}}$ be 
$2t$ complex conjugates of $\alpha$. 
Then, $d=s+2t$. We set 
$$
 D = \max\left\{d^{-1}\cR_0(\al),\,  \gamma/(cd) \right\}.
$$
In particular, $|\al_i| \le \sqrt{cd D}$ for each $i=1, \ldots,  s+t $. 
Thus,  the distance between any pair of conjugates of $\alpha$ is at most $2 \sqrt{cdD}$. 

Then, applying the estimate of Lemma~\ref{lem:schur} to the 
product of distances between $s_0$ real conjugates of $\alpha$ counted in $\cR_0(\al)$ (with $L=\cR_0(\al) \leq dD$) and the trivial bound 
$$
| \al_i - \al_j|^2 \le 4c dD
$$
for each of  the other 
$$
\frac{d^2-d}{2}-\frac{s_0^2-s_0}{2}
$$
distances between conjugates of $\al$ 
 in the product formula for the discriminant $\Delta(\alpha)$ of $\alpha$, we obtain
$$
 |\Delta(\alpha)| \leq (4cdD)^{(d^2-d)/2-(s_0^2-s_0)/2} \(\frac{dD}{s_0^2-s_0}\)^{(s_0^2-s_0)/2}  \prod_{k=1}^{s_0} k^k.
$$
Applying Lemma~\ref{lem:prod} and collecting the powers of $D$ together,  after some simple calculations,  we obtain
\begin{align*}
 |\Delta(\alpha)| & \leq  (4cd)^{(d^2-d)/2-(s_0^2-s_0)/2}D^{(d^2-d)/2}   \\ 
 & \qquad\qquad\qquad\qquad\quad \(\frac{d}{s_0^2-s_0}\)^{(s_0^2-s_0)/2}   s_0^{(s_0^2+s_0)/2+1/12}   e^{-s_0^2/4 +O(1)}\\
 &= (4cd)^{(d^2-d)/2-(s_0^2-s_0)/2}D^{(d^2-d)/2} \\
 & \qquad\qquad\qquad\qquad\quad \(\frac{d}{s_0-1}\)^{(s_0^2-s_0)/2}  s_0^{s_0+1/12}   e^{-s_0^2/4 +O(1)}.
\end{align*}
We note that 
$$
\frac{d}{s_0-1} = 1 + O(1/\log s_0)= \exp\(O(1/\log s_0)\)
$$
and 
$$
\log d = \log s_0 + \log (d/s_0) =  \log s_0 + O(1),
$$
due to the choice of $s_0$. 
Therefore, from the previous upper bound on $|\Delta(\al)|$ we further derive that 
\begin{equation}
\label{eq:Clean}
 |\Delta(\alpha)| \leq  D^{(d^2-d)/2}  \exp\(F(s_0,d)\),  
\end{equation}
where (using $d-s_0 = O(s_0/\log s_0)$) we have 
\begin{equation*}
\begin{split}
F(s_0,d) &=\frac{(d^2-d)-(s_0^2-s_0)}{2} \log (4cd) - s_0^2/4 + O(s_0^2/\log s_0 )\\
 &=\frac{1}{2}(d^2-s_0^2)  \log s_0 - s_0^2/4 + O(s_0^2/\log s_0)\\
  &=s_0(d-s_0)  \log s_0 - s_0^2/4 + O(s_0^2/\log s_0). 
\end{split}
\end{equation*}
Since $d-s_0\le (1/4-\eta)s_0/\log s_0$, we see that
\begin{equation*}
\begin{split}
F(s_0,d) & \le -\eta s_0^2  + O(s_0^2/\log s_0)=  -\eta (d^2-d)  + o(s_0^2) \\
& \le - \kappa(d^2-d) , 
\end{split}
\end{equation*}
provided that $s_0$ is large enough. 
Now, recalling~\eqref{eq:Clean} we derive the inequality 
$|\Delta(\al)|^{2/(d^2-d)}\le D e^{-2\kappa}$, which yields
the desired result. 
\end{proof}

Since  $|\Delta(\alpha)| \ge 1$ for any algebraic integer $\al\ne 0$, by Theorem~\ref{thm:manyreal alpha} (taking $c=2,\eta=\eps, \kappa=3\eps/4$ there), we immediately derive:

\begin{cor}
\label{cor:manyreal alpha}
 For each  $\eps \in (0,1/4)$, there is a positive integer $S_0(\eps)$ such that
for every algebraic integer  $\alpha$ of degree 
$d$ over $\Q$ we have 
$$
\max\left\{\cR_0(\al),\, \gamma/2\right\} \ge d e^{3\eps/2}.
$$
where $s_0,\cR_0(\al)$ and $\gamma$ are defined as in Theorem~\ref{thm:manyreal alpha}, $s_0 \geq S_0(\eps)$  
and 
$d-s_0<(1/4-\eps) s_0/\log s_0$. 
\end{cor}

We are now able to establish the main result of this subsection. 

\begin{theorem}
\label{thm:manyreal K}
For every number field $\K$ of signature $(s,t)$, where $s$ 
is sufficiently large 
and
$$ t\le 0.096 \sqrt{ s/\log s}, 
$$ 
we have $m(\K)=1$.
\end{theorem}

\begin{proof}
We  show that for each $\eps \in (0,1/4)$, if $s$ is sufficiently large 
 and
\begin{equation}\label{astuoni}
t\le \sqrt{\eps}(1/2-2\eps) \sqrt{ s/\log s}, 
\end{equation}
we have $m(\K)=1$ for every number field $\K$ with signature $(s,t)$. 
Note that the factor $\sqrt{\eps}(1/2-2\eps)$ attains its maximum value at $\eps=1/12$. 
Then, choosing $\eps = 1/12$ and using 
$$
\sqrt{1/12}(1/2-1/6)=1/(6\sqrt{3})= 0.096225\ldots, 
$$
we obtain the desired result.

Suppose that the minimum of $m_{\K}(\al)$ is attained at an algebraic integer $\al \in \K$ of degree $d_1=s_1+2t_1$ with $s_1$ real conjugates and $2t_1$ complex conjugates. 
Since the example $1 \in \K$ shows that $m(\K) \leq 1$, 
for the reverse inequality $m(\K) \geq 1$
it suffices to show that $m_\K(\al) \geq 1$.
By Corollary~\ref{cor:manyreal alpha},  
there is nothing to prove if $\K=\Q(\al)$ (provided that $s$ is large enough).
So,
in the following, we assume that $\Q(\al)$ is a proper subfield of $\K$. 

Denote $d=s+2t$, which is the degree of $\K$. Let $d_2 = [\K:\Q(\al)]$. Then, $d_2=d/d_1 >1$. 
Let $\al_1,\ldots,\al_{s_1}$ be all the real conjugates of $\al$. 
Assume that for each real $\al_i,1\le i \le s_1$, it appears $u_i$ times under the $s$ real embeddings of $\K$ and $2v_i$ times under the $2t$ complex embeddings of $\K$. Here, we note that $u_i+2v_i=d_2$ for each $i$. We also have 
\begin{equation} \label{eq:sutv}
s=u_1+ \ldots + u_{s_1} \mand 
t = d_2t_1 + v_1 + \ldots + v_{s_1}. 
\end{equation} 
So, in view of \eqref{bhyt} we can write 
\begin{equation} \label{eq:mKa}
(s+t) m_{\K}(\al) = \sum_{i=1}^{s_1} (u_i+v_i)\al_i^2 + d_2 \cC(\al).
\end{equation} 

We first assume that $t_1 > 0$. 
Let $k$ be the number of $i =1, \ldots, s_1$ with  $v_i>0$. Then, the number of  $i =1, \ldots, s_1$ with  $v_i=0$ 
is $s_0 = s_1 -k$, and for these $i$ we have $u_i=u_i+2v_i=d_2$. 
By~\eqref{eq:sutv} and $t_1>0$, we obtain 
$$
d_2 \le t, \qquad t_1 \le t/2 \mand k < t.
$$
Besides, \eqref{eq:sutv} implies that $s\le d_2s_1$. 
So, we have 
$$
s_1 \ge s/d_2 \ge s/t. 
$$
Thus,  under the condition~\eqref{astuoni} we have
\begin{equation}
 \label{eq:s0 large}
\begin{split}
s_0  & = s_1 - k > s_1 - t \ge s/t - t   = (1 + o(1)) s/t\\
&  \ge 
\( \frac{1}{\sqrt{\eps}(1/2-2\eps)} +o(1)\) \sqrt{ s \log s} >10\sqrt{s\log s} 
\end{split}
\end{equation}
for $s$ large enough.
Also, using $\varepsilon <1/4$ we obtain
\begin{align*}
d_1 - s_0 &= s_1+2t_1-s_0=2t_1 + k \le t + k  < 2t \\
& \le 2 \sqrt{\eps}(1/2-2\eps) \sqrt{ s/\log s} \le 
(1/2-2\eps)\sqrt{s/\log s}. \\
\end{align*} 
Now, it is also easy to verify that combining this with \eqref{eq:s0 large} yields the inequality
\begin{equation}
 \label{eq:d1-:s0 small1}
d_1 - s_0   < (1/2-2\eps)s_0/(10\log s_0)<(1/4 - \eps) s_0 / \log s_0, 
\end{equation}
provided that $s$ is large enough.

Let $\cR_0(\al)$ be the sum of $s_0$ squares of such $\al_i, 1\le i \le s_1$, which are counted in the sum, where $v_i=0$. For the real conjugates $\al_i$ of $\al$ not counted in $\cR_0(\al)$ we have $u_i+v_i \ge (u_i+2v_i)/2=d_2/2$. Let us denote by $\cR_1(\al)$ the sum of squares of these real conjugates. 
Now, combining~\eqref{eq:mKa} and~\eqref{eq:d1-:s0 small1} with Corollary~\ref{cor:manyreal alpha}, we get 
\begin{align*}
(s+t)m_{\K}(\al) & \ge d_2 \cR_0(\al) +  \frac{d_2}{2}\cR_1(\al) + d_2 \cC(\al) \\
& \ge d_1d_2 e^{3\eps/2} > d_1d_2 = d =s+2t \ge s+t, 
\end{align*}
which actually concludes the proof of the case $t_1 >0$.

In the following, we assume that $t_1=0$. Then, $d_1=s_1$ and $\cC(\al)=0$.  
Note that since $t_1=0$, we can assume that $s_1\ge 2$, because otherwise $\al$ is a non-zero rational integer and we are done. 

Let us first recall a result due to Smyth~\cite[Theorem~1]{Smyth2} on the mean values of totally real algebraic integers, which asserts that 
if $\be$ is a totally real algebraic integer of degree $q\ge 2$ and $\{\be_1,\ldots,\be_q\}$ is the full set of conjugates of $\be$, then 
\begin{equation} \label{eq:Smyth}
\be_1^2+\ldots + \be_{q}^2 \ge 3q/2, 
\end{equation}
where the equality is achieved only when $\be$ is the root of the polynomial $x^2 \pm x-1$.

We claim that for each $i, 1\le i \le s_1$, we have 
\begin{equation} \label{eq:u+v}
u_i + v_i > \frac{2(s+t)}{3s_1}.
\end{equation}
If so, then combining~\eqref{eq:mKa} with~\eqref{eq:Smyth} and~\eqref{eq:u+v}, we obtain  
$$
(s+t)m_{\K}(\al) > \frac{2(s+t)}{3s_1} 
\sum_{i=1}^{s_1} \al_i^2  \ge \frac{2(s+t)}{3s_1} \frac{3s_1}{2} = s+t,  
$$
which implies that $m_{\K}(\al) > 1$. 

To prove the inequality \eqref{eq:u+v}, we assume that there exists $j$ such that $u_j + v_j \le  2(s+t)/(3s_1)$. 
Notice that since $t_1=0$, we have $d=s+2t=s_1d_2$. 
Then, 
\begin{align*}
d_2 = u_j + 2v_j & \le 2(s+t)/(3s_1) + v_j \\
& = 2d_2(s+t)/(3(s+2t)) + v_j \\
& \le d_2 - d_2/3 + v_j, 
\end{align*}
which together with~\eqref{eq:sutv} yields that 
$$
t \ge v_j \ge d_2/3 = (s+2t) / (3s_1) \ge s /(3s_1). 
$$ 
So, we get 
$$
s \le 3ts_1 < \frac{3}{10} s_1 \sqrt{s/\log s} < s_1 \sqrt{s}.
$$
In particular, we now conclude that if $s$ is large enough than so is $s_1$.
To continue with the case when $t_1=0$  and $s_1$ is sufficiently large,
 we define $k$ to be the number of $i =1, \ldots, s_1$ 
with  $v_i \ge \eps d_2/2$. 

If $k=0$, then for each $i$ with $1\le i \le s_1$ we have $v_i < \eps d_2/2$ and $u_i = d_2 - 2v_i > (1-\eps)d_2$. Observing that  $d_1 - s_1=0$, and thus the conditions of Corollary~\ref{cor:manyreal alpha} are trivially satisfied, 
and  applying~\eqref{eq:mKa} , we deduce that  
\begin{align*}
(s+t)m_{\K}(\al)  > (1-\eps)d_2 \cR(\al) 
 \ge (1-\eps)e^{3\eps/2}d_1d_2. 
\end{align*}
In view of $(1-\eps)e^{3\eps/2}>1$ (because $0< \eps < 1/4$), this is greater than $d_1d_2=d=s+2t>s+t$, and we are done in this case.   

Next, we suppose that $k \geq 1$. 
Then, by~\eqref{eq:sutv} with $t_1=0$, we get $k \le t$ and $t \ge  k \eps d_2/2$, and so 
$$
d_2\le 2t/(k\eps) \le 2t/\eps.
$$ 
Let $s_0$ be the number of   $i =1, \ldots, s_1$  for which we have $v_i < \eps d_2/2$.  Clearly, $s_0=s_1-k$, and for each of those $i$ we have
\begin{equation} \label{eq:uvi}
u_i + v_i = d_2 - v_i > (1-\eps/2) d_2.
\end{equation} 
Observe that, by~\eqref{astuoni}, 
$$
s_1 =\frac{s+2t} {d_2} \ge \frac{s}{d_2} \ge \frac{\eps s}{2t} \geq  \frac{\sqrt{\eps s \log s}}{1-4\eps}. 
$$
Thus,
$$
s_0 = s_1 - k \ge s_1 - t \geq \frac{\sqrt{\eps s \log s}}{1-4\eps} -t. 
$$ 
Since $t<0.1\sqrt{s/\log s}$, this yields $s_0>(1+\eps)\sqrt{\eps s \log s}$ for $s$ sufficiently large. 
Hence, 
\begin{equation}
\begin{split}
 \label{eq:d1-:s0 small2}
d_1 - s_0 = 2t_1 + k = k \le t 
& \le \sqrt{\eps}(1/2 -2\eps)\sqrt{s/\log s} \\
& < (1/4 - \eps) s_0 / \log s_0 
\end{split} 
\end{equation} 
provided that $t$ satisfies~\eqref{astuoni}  and $s$ (and thus $s_0$) are large enough. 

Let $\cR_0(\al)$ and $\cR_1(\al)$ be defined as the above (namely, the first one is the sum of $s_0$ squares of real conjugates, and the second is the sum of the remaining $s_1-s_0$ squares of real conjugates). Then, combining~\eqref{eq:mKa} and~\eqref{eq:uvi} with Corollary~\ref{cor:manyreal alpha}, which 
applies due to~\eqref{eq:d1-:s0 small2},
 we deduce
\begin{align*}
(s+t)m_{\K}(\al) & \ge (1-\eps/2)d_2 \cR_0(\al) +  \frac{d_2}{2}\cR_1(\al) \\
& \ge (1-\eps/2)e^{3\eps/2}d_1d_2 > d_1d_2 =s+2t\ge s+t, 
\end{align*}
 which completes the proof. 
\end{proof} 

We want to emphasize that for the element $\al \in \K$ in the above proof, if $\al \not\in \Z$, then actually we have shown that $m_{\K}(\al)>1$, which however contradicts   the choice of $\al$. 
Hence, either $\al=\pm 1$ or $\K=\Q(\al)$.  

It is not difficult to see that  the arguments used in the proofs of Theorems~\ref{thm:manyreal alpha}
and~\ref{thm:manyreal K} can be made completely explicit. In turn, this can be used to 
obtain an explicit value of $s$ beyond which 
the inequality of Theorem~\ref{thm:manyreal K}  holds. 

\subsection{Fields with  many real and many complex embeddings}
\label{sec:large t} 

In Section~\ref{sec:small t} we have shown that $m(\K)=1$ provided that $t$ is small compared to $s$.
Here we show how to construct series of fields with $m(\K) < 1$ when $t$ and $s$
are of comparable sizes. 

We first make a preparation as follows. 
\begin{theorem}\label{cubic}
 Let $\alpha=\alpha_1$ be a nonzero real cubic algebraic integer with two complex conjugates $\alpha_2$ and $\alpha_3=\overline{\alpha_2}$.  
Then, 
\begin{equation}\label{kubas}
\alpha^2 + |\alpha_2|^2 \geq \vartheta^{-2}+\vartheta=1.894558\ldots,
\end{equation} 
where
$\vartheta=1.324717\ldots$ is the real root of $x^3-x-1$, and the equality
in~\eqref{kubas} is attained 
if and only if $\alpha=\pm \vartheta^{-1}$. 
\end{theorem}

\begin{proof}
If $\alpha$ is not an algebraic unit, by Lemma~\ref{lem:nece} we have $m(\al) \ge 1$, and thus $\|\alpha\|^2=\alpha^2+|\al_2|^2 \geq 2$. 

Let $\al$ be an algebraic unit. Then,
by the result of Smyth~\cite{Smyth}, 
 the Mahler measure $M(\alpha)$ of a non-reciprocal algebraic integer $\alpha$ is at least
$\vartheta$.  Furthermore, for real cubic algebraic integers, this value is attained only for 
$\alpha=\pm \vartheta, \pm \vartheta^{-1}$. Assume first $|\al|>1$. Then, 
$M(\al)=|\al|$ and $|\al_2|^2=\al_2\al_3=|\al|^{-1}=M(\al)^{-1}$,
so that
\[\al^2+|\al_2|^2=M(\al)^2+M(\al)^{-1} \geq 2 \sqrt{M(\al)}>2,\]
which is stronger than~\eqref{kubas}.

Now, assume that $|\al|<1$. Then, $M(\al)=\al_2 \al_3=|\al_2|^2$
and $|\al|=M(\al)^{-1}$. As the function $f(x)=x^{-2}+x$ is increasing in the interval $[{2}^{1/3}, +\infty)=[1.259921\ldots,+\infty)$, we have $f(x) \geq
f(\vartheta)= \vartheta^{-2}+\vartheta$
for each $x \geq \vartheta=1.324717\ldots$. Consequently, 
\[\al^2+|\al_2|^2=M(\al)^{-2}+M(\al) =f(M(\al))\geq f(\vartheta)= \vartheta^{-2}+\vartheta,\]
which yields~\eqref{kubas}.
This proves the theorem, since when $|\al|<1$, the equality $M(\al) =\vartheta$ for cubic $\al$ holds if and only if $\al=\pm \vartheta^{-1}$. 
\end{proof}

Taking, for instance, $s=t$ we find from~\eqref{eq:Bound mK}
that 
$$
m(\K) \geq 2^{-5/3}+2^{-2/3}=3\cdot 2^{-5/3}=0.944940\ldots.
$$
This lower bound is close to being sharp. Indeed,
selecting the field $\K=\Q(\vartheta,\beta)$, where $\beta$
is a totally real algebraic number of degree $s$ over $\Q(\vartheta)$ (and also over $\Q$), by~\eqref{cfrt1} we see that
$\K$ has signature $(s,s)$. 
This gives $\|\vartheta^{-1}\|_{\K}^2=s(\vartheta+\vartheta^{-2})$ by~\eqref{eq:alpha RC}, namely, 
we have:

\begin{theorem}\label{cubic1}
 For each $s \geq 2$, there exist infinitely many number fields $\K$ of signature $(s,s)$ for which 
 \[
0.944940\ldots =3\cdot 2^{-5/3} \le 
 m(\K) \leq 
\frac{\vartheta+\vartheta^{-2}}{2}=0.947279\ldots.
\] 
\end{theorem}

\begin{proof}
From the above discussion, we only need to show that there are infinitely many totally real number fields of given degree. In fact, this is a classical result; see, for example,~\cite[Theorem~1]{ABC} or~\cite[Theorem~1.2]{DS}. 
\end{proof}

\subsection{Fields with few real embeddings}
\label{sec:small s}

The next result implies that $m(\al)<1$ for some algebraic integers $\al$ of arbitrarily large degree with one or two real conjugates. 

\begin{theorem}\label{bhu1}
Let $\al$ be any root of $x^n+x^{n-1}+\ldots+x-1=0$, where $n\ge 2$. Then, the field $\K=\Q(\al)$ is of degree $n$ over $\Q$, the signature of $\K$ is $(s,t)$, $s+2t=n$, where $s=1$ for $n$ odd and $s=2$ for $n$ even, and 
$$
\|\al\|^2 
=s+t-3/4+\log 2+O(n^{-1/4}).
$$
\end{theorem}

In particular, since $3/4-\log 2>0$, for such fields $\K$
with $s \in \{1,2\}$ Theorem~\ref{bhu1} implies that  the inequality $m(\K)<1$ holds for all sufficiently large $n$.

\begin{proof}
 Put $t=(n-1)/2$ for $n$ odd and $t=(n-2)/2$ for $n$ even. By Lemma~\ref{bhu0}, the signature of the field $\K=\Q(\al)$, where without restriction of generality we can take the positive root $\al=\vartheta_n^{-1}$, 
is $(n-2t,t)=(s,t)$. 
It remains to evaluate $\|\al\|^2$. 

From~\eqref{didz} one can easily derive that
\[
\cR(\al)=|\al|^2=1/4+O(1/n)
\]
for $n$ odd (when $s=1$). Using in addition~\eqref{didz1} we further find that
\[
\cR(\al)=5/4+O(1/n)
\]
for $n$ even (when $s=2$). 
Thus, in both cases, $s=1$ and $s=2$, we can write
\begin{equation}\label{Qrrrr}
\cR(\al)=s-3/4+O(1/n).
\end{equation}

On the other hand, employing Lemma~\ref{lopo} with $q=2$
we find that $$\cC(\al)=t+\log 2+O(n^{-1/4}).$$
Combining this with~\eqref{Qrrrr} 
we now find that
\[\|\al\|_{\K}^{2} =\|\al\|^2=\cR(\al)+\cC(\al)=s+t-3/4+\log 2+O(n^{-1/4}),\]
where $\K=\Q(\al)$. This gives the desired result. 
\end{proof}

We continue to give examples in the following. 

\begin{theorem}\label{kiy}
Suppose $n=4k+2$, $k \ge 1$, and  $\al$ is a root of 
the polynomial $x^n+x^{n-2}+\ldots+x^2-1$. Then, 
$\al$ has degree $n$ over $\Q$ and has exactly two real conjugates
and its absolute size is
$$\|\al\|^2=n/2+\log 2+O(n^{-1/4}).$$
\end{theorem}

\begin{proof} 
Using the notation in Lemma~\ref{bhu0}, let $\vartheta_j, j=1,\ldots,n/2$ be the roots of the polynomial $x^{n/2}-x^{n/2-1}-\ldots-x-1$ such that $\vartheta_{n/2}>1$ is the unique real root and $\vartheta_j^{-1}$, $j=1,\ldots,(n-2)/4$ have positive imaginary parts. 
By Lemma~\ref{auxil2}, the polynomial 
$x^{n}+x^{n-2}+\ldots+x^2-1$ is irreducible, so without restriction of generality we can assume that $\al=1/\sqrt{\vartheta_{n/2}}$ by using the reciprocity of polynomials. 
This $\al$ has a real conjugate $-\al$, and $n-2$ complex conjugates  $\pm 1/\sqrt{\vartheta_j}$, $j=1,\ldots,n/2-1$. Hence, by~\eqref{didz},
$$
\cR(\al)=2\al^2=2\vartheta_{n/2}^{-1}=1+O(1/n).$$
Similarly, 
$$\cC(\al)=2\sum_{j=1}^{(n-2)/4} |\vartheta_j|^{-1}.$$
Note that, $\vartheta_j^{-1}$, $j=1,\ldots,(n-2)/4$, 
are the complex roots of the polynomial $x^{n/2}+x^{n/2-1}+\ldots+x-1$
with positive imaginary parts. 
Hence,
by Lemma~\ref{lopo} with $q=1$ and $t=(n-2)/4$,
we find that
$$
\cC(\al)=2\(\frac{n-2}{4}+\frac{\log 2}{2}+O(n^{-1/4})\) = n/2-1+\log 2+O(n^{-1/4}).
$$
Combining this with $\cR(\al)$ we get the desired result. 
\end{proof}

For the algebraic integer $\al$ defined in Theorem~\ref{kiy}, we have 
$$
m(\al) = \frac{\|\al\|^2}{(n+2)/2} 
= 1 - 2(1-\log 2)/(n+2) + O(n^{-5/4}). 
$$ 
Thus, for all sufficiently large $n$ we obtain $m(\al)<1$, which implies that $m(\K)<1$ for $\K=\Q(\al)$. 

\begin{theorem}\label{kiy1}
Suppose that $s \geq 2$ is an even integer. Then, for each integer $k \ge 1$ there is a field $\K$
of degree $n=(2k+1)s$ with signature $(s,(n-s)/2)$  
for which 
$$d(\Lambda_{\K})^2 \leq \frac{n}{2}+\frac{s\log 2}{2}+O(s^{5/4}n^{-1/4}).$$
\end{theorem}

\begin{proof}
Write $s=2s_1$ and take $\K=\Q(\al,\be)$, where $\al$
is the algebraic number of degree $n/s_1=4k+2$ with $2$ real conjugates as that in Theorem~\ref{kiy} and $\be$ is a totally real algebraic number of degree $s_1$ over $\Q(\al)$ and also over $\Q$. Then, by~\eqref{cfrt1} of Lemma~\ref{subfield}, the signature of the field $\K$ is $(s,(n-s)/2)$. Furthermore, using~\eqref{cfrt3} and Theorem~\ref{kiy} we deduce that
\begin{align*}\|\al\|_{\K}^2 &= s_1\|\al\|^2 =s_1
\frac{n}{2s_1}+s_1\log 2+O(s_1 (n/s_1)^{-1/4})\\&=\frac{n}{2}+\frac{s\log 2}{2}+O(s^{5/4}n^{-1/4}),
\end{align*}
which implies the required result.
\end{proof}

In particular,  we see that  for the field $\K$ of Theorem~\ref{kiy1} we have
$$
s+t  = n\frac{k+1}{2k+1}\mand \frac{n}{2}+\frac{s\log 2}{2} =
n  \frac{2k+1 + \log 2}{2(2k+1)}
$$
and 
\begin{equation}
m(\K) \le 1  -  \frac{1- \log 2}{2(k+1)}  + O(k^{-5/4}).
\end{equation}
With fixed $s$ we get fields $\K$ with $m(\K) < 1$ for all sufficiently large $k$. 

Based on Theorem~\ref{cubic} and in view of the strategy in Theorem~\ref{thm:al 1/n}, we can get more number fields $\K$ with few real embeddings such that $m(\K)<1$. 

\begin{theorem}\label{cubic2}
Given a positive integer $n$, let $\al = \vartheta^{-1/n}$, where $\vartheta=1.324717\ldots$ is the root of $x^3-x-1=0$, and $\K=\Q(\al)$. 
Then, 
$$
m(\al) < 1, 
$$ 
and, in particular, $m(\K) < 1$.
\end{theorem}

\begin{proof} Note that $\al$ is a root of the polynomial
$x^{3n}+x^{2n}-1$. If it were reducible then its reciprocal polynomial 
$-(x^{3n}-x^n-1)$ must be reducible too. This is, however, not the case by an old result of 
Ljunggren~\cite[Theorem~3]{lju}. 
Hence, $[\Q(\al):\Q]=3n$. 

First, suppose that $n$ is odd. Then,
$\K$ has signature $(s,t)$ with $s=1$ and 
$$
t= \frac{1}{2}(3n-1). 
$$
Consequently,
\begin{equation}
\label{eq:s+t 2}
s + t = \frac{1}{2}(3n+1). 
\end{equation}

Note that the algebraic integer $\al$ has one real conjugate of modulus $\al$, $(n-1)/2$ pairs of complex conjugates of the same modulus $\al$, and $n$ pairs of complex conjugates of the same modulus $\al^{-1/2}$. 
Thus, we obtain 
$$
\cR(\al)=\al^2= \vartheta^{-2/n}, 
$$
and 
$$
\cC(\al) = \frac{n-1}{2}\al^2 + n \al^{-1} 
= \frac{n-1}{2}\vartheta^{-2/n} + n\vartheta^{1/n}.
$$
Therefore,
$$\|\al\|^2=\cR(\al)+\cC(\al) = \frac{n+1}{2} \vartheta^{-2/n}+n\vartheta^{1/n}.$$
In view of~\eqref{eq:s+t 2}, it remains to check that 
\begin{equation}\label{patikr}
\frac{n+1}{2}y^{-2}+ny<\frac{3n+1}{2}
\end{equation}
for $y=y_0=\vartheta^{1/n}$. 

We now note that  for $n=1$ the inequality~\eqref{patikr} 
holds by  Theorem~\ref{cubic} (see~\eqref{kubas}). 
For $n \geq 3$, since the function $f(y)=(n+1)y^{-2}/2+ny$ is decreasing in the interval 
$(1,y_1)$, where $y_1=(1+1/n)^{1/3}$, and 
satisfies $f(1)=(3n+1)/2$, in order to prove~\eqref{patikr} we need to show that $1<y_0<y_1$. To verify this it suffices to check that $$3\log \vartheta
<n\log (1+1/n)$$ for $n \geq 3$.  One can easily see that the right hand side of this inequality is increasing in $n$, and its smallest value  
$3 \log (1+1/3)=0.863046\ldots$ is greater than $3\log \vartheta=0.843598\ldots$. 
This concludes the proof in the case when $n$ is odd. 

Next, assume that $n \geq 2$ is even. Then, $s=2$, $t=3n/2-1$,
$\cR(\al)=2\al^2= 2\vartheta^{-2/n}$,
and 
$$
\cC(\al) = (n/2-1)\al^2 + n \al^{-1} 
= (n/2-1)\vartheta^{-2/n} + n\vartheta^{1/n}.
$$
Therefore, instead of~\eqref{patikr} we now need to verify the inequality
$$\|\al\|^2=\cR(\al)+\cC(\al) = \frac{n+2}{2} y^{-2}+ny<s+t=\frac{3n+2}{2}$$
for $y=y_0=\vartheta^{1/n}$.

For $n \geq 2$, since the function $g(y)=(n/2+1)y^{-2}+ny$ is decreasing in the interval 
$(1,y_2)$, where $y_2=(1+2/n)^{1/3}$, and 
satisfies $g(1)=(3n+2)/2$, as above it suffices to show that $1<y_0<y_2$. Now, to verify this we need to check that $$3\log \vartheta
<n\log (1+2/n)$$ for $n \geq 2$.  Again,  the right hand side of this inequality is increasing in $n$, and its smallest value 
$2 \log (1+2/2)= 1.386294\ldots$ is greater than 
$3\log \vartheta=0.843598\ldots$. 
This completes the proof. 
\end{proof}

Using the examples in Section~\ref{sec:num} below and following the method in Theorem~\ref{cubic2}, one can construct more number fields $\K$ with $m(\K)<1$.

\section{Numerical examples}
\label{sec:num}

\subsection{Preliminaries}
In this section, we use the computer algebra system PARI/GP~\cite{Pari}
to make some computations for number fields of low degree. 

Let $\al$ be an algebraic integer of degree $n$, which has $s$ real conjugates and $2t$ complex conjugates. If $m(\al) < 1$ then the absolute value of each conjugate of $\al$ is less than $(s+t)^{1/2}$, and so the minimal polynomial $f(x)$ of $\al$ has the form of 
\begin{equation}
\label{eq:pol}
x^n + a_1x^{n-1} + \ldots + a_{n-1} x + a_n \in \Z[x],
\end{equation}
where 
\begin{equation}
\label{eq:coeff bound}
|a_i| < \binom{n}{i}(s+t)^{i/2} \le \binom{n}{i}n^{i/2}, \qquad  i=1,2, \ldots, n.
\end{equation}
This means that for fixed integer $n\ge 1$, we can find all algebraic integers $\al$ of degree $n$ satisfying $m(\al)<1$ by testing finitely many monic integer polynomials. Besides, this also enables us to find the shortest non-zero vectors in such lattices $\Lambda_\K$ with $m(\K)<1$. However, the computational complexity increases very fast when $n$ becomes large.

We especially want to point out that~\eqref{eq:coeff bound} implies that for any fixed integer $n\ge 1$ there are only finitely many algebraic integers $\al$ of degree $n$ such that $m(\al)<1$. However, there can be infinitely many number fields $\K$ of degree $n$ for which $m(\K)<1$; see Theorem~\ref{cubic1}.

\subsection{Numerical tables}

Given an irreducible polynomial $f$ in $\Z[x]$, define $m(f)=m(\al)$, where $\al$ is an arbitrary root of $f$. We say that $f$ has signature $(s,t)$ if $f$ has exactly $s$ real roots and $2t$ complex roots, and so $\deg f = s+2t$.   
Based on the above preparations and using PARI/GP, one can list all monic irreducible polynomials $f$ of low degree for which $m(f)<1$. Here, we do this for polynomials of degree at most $6$. 

We have indicated before that for any monic irreducible polynomial $f \in \Z[x]$ of signature $(s,t)$, where $st=0$, we must have $m(f) \ge 1$. So, it suffices to consider the cases when $st\ne 0$.  

In Table~\ref{deg34}, we list all monic irreducible polynomials $f$ of degree 3 and 4 such that $m(f)<1$. The first column stands for the signature of a polynomial. The polynomials $f$ themselves and their $m(f)$ are presented in the second and third columns, respectively, and for comparison we illustrate the lower bound coming from~\eqref{eq:Bound mK} in the last column.

\begin{table}[H]
\centering
\begin{tabular}{|c|c|c|c|}
\hline
$(s,t)$ & Polynomials $f$ & $m(f)$ & Lower bound   \\ \hline
\multirow{4}{*}{$(1,1)$}  & $x^3-x^2+1$ & \multirow{2}{*}{0.947279\ldots}  & \multirow{4}{*}{0.944940\ldots}   \\ 
& $x^3+x^2-1$ &   &     \\ \cline{2-3}
& $x^3+x-1$ & \multirow{2}{*}{0.965571\ldots}  &   \\ 
& $x^3+x+1$ &   &    \\  \hline
\multirow{3}{*}{$(2,1)$} & $x^4+x^2-1$ & 0.951367\ldots &   \\ \cline{2-3}
& $x^4-x^3+x^2+x-1$ & \multirow{2}{*}{0.979971\ldots} &  0.942809\ldots \\ 
& $x^4+x^3+x^2-x-1$ &  &  \\ \hline
\end{tabular}
\caption{Monic irreducible polynomials of degree $3, 4$ and $m(f)<1$}
\label{deg34}
\end{table}

Our computations show that there are $59$ monic irreducible polynomials $f$ of degree $5$ or $6$ with $m(f)<1$. In Table~\ref{deg56},  in the second column we give the number of such polynomials for each signature, and only list the polynomials with minimum value of $m(f)$ among such $f$.  

\begin{table}[H]
\centering
\begin{tabular}{|c|c|c|c|c|}
\hline
$(s,t)$ & Number & Polynomials $f$ & $m(f)$ & Lower bound  \\ \hline
\multirow{2}{*}{$(1,2)$} & \multirow{2}{*}{22} & $x^5-x^3-x^2+x+1$ & \multirow{2}{*}{0.961783\ldots} &  \multirow{2}{*}{0.957248\ldots}  \\ 
& & $x^5-x^3+x^2+x-1$ &  &   \\ \hline
$(3,1)$ & 0 & & &   \\ \hline 
$(2,2)$ & 37 & $x^6+x^2-1$ & 0.946467\ldots & 0.944940\ldots  \\ \hline
$(4,1)$ & 0 & & &   \\ \hline
\end{tabular}
\caption{Monic irreducible polynomials of degree $5, 6$ with $m(f)<1$}
\label{deg56}
\end{table}

\subsection{Examples of fields with $m(\K)<1$ and $m(\K)=1$}

We see from  Lemma~\ref{lem:nece} that 
for small degree $n$ the condition $m(\al)<1$ implies that $\al$ is an algebraic unit, 
that is, $a_n=\pm 1$ in~\eqref{eq:pol}.
Hence, it might be of interest to find an algebraic integer $\al$, which is not an algebraic unit, but satisfies $m(\al)<1$. 
At the present time we do not have such examples; see also Question~\ref{quest:ma<1}
below. 

In the following we present some direct consequences of Tables~\ref{deg34} and~\ref{deg56}.

\begin{theorem}
\label{thm:cubic}
Let $\K$ be a number field  of degree 3 with signature $(s,t) =(1,1)$. 
Then, $m(\K) = 1$ if and only if $\K$ does not contain any roots of the following polynomials:
\begin{align*}
& x^3 - x^2 + 1, \quad x^3+x^2-1, \\
& x^3 + x - 1, \quad x^3 + x + 1. 
\end{align*}
\end{theorem}

\begin{theorem}
\label{thm:quartic}
Let $\K$ be a number field of degree 4 with signature $(s,t)=(2,1)$. 
Then, $m(\K) = 1$ if and only if $\K$ does not contain any roots of the following polynomials:
\begin{align*}
& x^4 + x^2 - 1, \quad 
 x^4 - x^3 + x^2 + x - 1, \\
& x^4 + x^3 + x^2 - x - 1. 
\end{align*}
\end{theorem}

\begin{proof} 
From Table~\ref{deg34}, it suffices to show that for any algebraic integer $\al \in \K$ of degree 2, we have $m_{\K}(\al)\ge 1$. Note that  $\K$ has only two complex embeddings, so $\al$ must be real. Let $\be$ be the conjugate of $\al$, and assume that the minimal polynomial of $\al$ is $x^2+bx+c\in \Z[x]$. So, we have 
\begin{equation} \label{eq:Vieta}
\al + \be  = -b, \qquad \al\be = c 
\mand b^2-4c >0. 
\end{equation}
By definition (or in view of~\eqref{eq:mKa}), we get 
$$
\| \al \|_{\K}^2 = 2\al^2 + \be^2 
\mor  \| \al \|_{\K}^2 = \al^2 + 2\be^2. 
$$
Thus, $\| \al \|_{\K}^2 \ge 2\sqrt{2}|\al\be|=2\sqrt{2}|c|$. Then, if $|c|\ge 2$, we have $\| \al \|_{\K}^2 >3$, which implies that $m_{\K}(\al)> 1$. 

Now assume that $|c|=1$. Clearly, $b\ne 0$.  
If $c=-1$, then $\al\be<0$, and we must have either $|\al|>|b|$ or $|\be|>|b|$, so we still obtain  $m_{\K}(\al)>1$ for $|b|\ge 2$; while for $b=\pm 1$, the inequality $m_{\K}(\al)>1$ can be verified by direct computations. Finally, suppose that $c=1$. Then, by~\eqref{eq:Vieta}, we have $|b|\ge 3$, and then either $|\al|>2$ or $|\be|>2$, and thus we also get  $m_{\K}(\al)>1$. This completes the proof. 
\end{proof}

We remark that the above proof of Theorem~\ref{thm:quartic} can be replaced by computations using PARI/GP. 
Indeed, if $\|\al\|_{\K}^2<2+1=3$, then $\al$ and $\be$ are of bounded absolute value, and so are the coefficients $b$ and $c$, thus we can use PARI/GP to check all these quadratic polynomials. 

\begin{theorem}
\label{thm:quintic}
Let $\K$ be a number field of degree 5. 
\begin{enumerate}
\item[(i)] Among number fields 
 $\K$  of signature $(s,t)=(1,2)$ the smallest value $m(\K) = 0.961783\ldots$ is achieved when $\K$ is generated by a root of  
 the following polynomials:
$$
x^5-x^3-x^2+x+1, \qquad x^5-x^3+x^2+x-1. 
$$

\item[(ii)] If $\K$ is of signature $(s,t)=(3,1)$ then
$m(\K) = 1$. 
\end{enumerate}
\end{theorem}

\begin{theorem}
\label{thm:sextic}
Let $\K$ be a number field of degree 6. 
\begin{enumerate}
\item[(i)]
Among number fields 
 $\K$  of signature  $(s,t)=(2,2)$ the smallest value $m(\K) = 0.946467\ldots$ is achieved when $\K$ is generated by a root of  
the following polynomial
$$
x^6 + x^2 -1. 
$$

\item[(ii)]
If $\K$ is of signature $(s,t)=(4,1)$ then
$m(\K) = 1$. 
\end{enumerate}
\end{theorem}

\begin{proof}
(i) 
First, for algebraic integers $\al$ of degree 6 with exactly two real conjugates, by 
Table~\ref{deg56}, we can find that 
the minimum of $\|\al\|$ is achieved only when $\al$ is a root of the polynomial $x^6+x^2-1$. 
It remains to consider algebraic integers $\al \in \K$ of degree 2 or 3. 

If $\al$ is of degree 2, then, since the signature $(s,t)$ is $(2,2)$, $\al$ must be real. 
Let $\be$ be the conjugate of $\al$. 
By~\eqref{eq:sutv} and~\eqref{eq:mKa}, we have 
$$
\| \al \|_{\K}^2 = 2\al^2 + 2\be^2 \ge 4|\al\be| \ge 4,
$$
which implies that $m_{\K}(\al)\ge 1$. 

Now, assume that $\al$ is of degree 3. 
Without loss of generality, we suppose that $\al$ is real. Let $\be,\gamma$ be the two conjugates of $\al$. If $\be,\gamma$ are complex, then by~\eqref{eq:sutv} and~\eqref{eq:mKa}, we obtain  
$$
\| \al \|_{\K}^2 = 2\al^2 + 2|\be|^2 +2|\gamma|^2 \ge 6|\al\be\gamma|^{2/3} \ge 6,
$$
which implies that $m_{\K}(\al)> 1$ again. 
On the other hand, If $\be,\gamma$ are real, then by~\eqref{eq:sutv} and~\eqref{eq:mKa}, we have 
$$
\| \al \|_{\K}^2 = 2\al^2 + \be^2 + \gamma^2
\mor  \| \al \|_{\K}^2 = \al^2 + 2\be^2 + \gamma^2 ,
$$
or
$$
\| \al \|_{\K}^2 = \al^2 + \be^2 + 2\gamma^2.
$$
If $\| \al \|_{\K}^2 < 2+2 =4$, then the absolute values of $\al,\be,\gamma$ are less than $2$, and so $\al$ must be a root of an irreducible cubic polynomial, which only has real roots, of the form  
$$
x^3+ax^2+bx+c \in \Z[x],
$$ 
where $|a|\le 5, |b|\le 11, |c| \le 7$. By checking all these polynomials using PARI/GP, we complete the proof of (i).

(ii) 
By Table~\ref{deg56}, we only need to consider algebraic integers $\al \in \K$ of degree 2 or 3. 

First, assume that $\al$ is of degree 2. Since the field $\K$ has only two complex embeddings, $\al$ must be real. Let $\be$ the conjugate of $\al$.  Then, by~\eqref{eq:sutv} and~\eqref{eq:mKa}, we have 
$$
\| \al \|_{\K}^2 = 2\al^2 + 3\be^2 
\mor  \| \al \|_{\K}^2 = 3\al^2 + 2\be^2. 
$$
Applying the same arguments as in the proof of Theorem~\ref{thm:quartic}, we obtain 
$m_{\K}(\al)>1$.  

Finally, assume that $\al$ is of degree 3. 
Similarly, all the conjugates of $\al$, say $\be$ and $\gamma$, must be real. Then, by~\eqref{eq:sutv} and~\eqref{eq:mKa}, we have 
$$
\| \al \|_{\K}^2 = \al^2 + 2\be^2 + 2\gamma^2
\mor  \| \al \|_{\K}^2 = 2\al^2 + \be^2 + 2\gamma^2 ,
$$
or
$$
\| \al \|_{\K}^2 = 2\al^2 + 2\be^2 + \gamma^2.
$$
If $\| \al \|_{\K}^2 < 4+1 =5$, then the absolute values of $\al,\be,\gamma$ are less than $\sqrt{5}$, and so $\al$ must be a root of an irreducible cubic polynomial, which only has real roots, of the form  
$$
x^3+ax^2+bx+c \in \Z[x],
$$ 
where $|a|\le 6, |b|\le 14, |c| \le 11$. By checking all these polynomials using PARI/GP, we conclude the proof. 
\end{proof}

We remark that, by Theorem~\ref{cubic1}, we have infinitely many sextic number fields $\K$ with signature $(2,2)$ and $m(\K) \leq 
(\vartheta+\vartheta^{-2})/2=0.947279\ldots$.

Note that, applying the same argument as in Theorem~\ref{cubic1} the example of Theorem~\ref{thm:sextic} implies that

\begin{theorem}\label{cubic11}
 For each even $s \geq 4$ there exist infinitely many number fields $\K$ of signature $(s,s)$ for which 
 \[
0.944940\ldots =3\cdot 2^{-5/3} \le 
 m(\K) \leq 
\frac{\zeta^2+\zeta^{-1}}{2}=0.946467\ldots,
\] 
where $\zeta=0.826031\ldots$ is the root of $x^6+x^2-1=0$. 
\end{theorem}

\begin{proof}
Begin with the field $\Q(\zeta)$ of degree $6$ and take a totally real algebraic number $\beta$ of degree $s/2$ over $\Q(\zeta)$ and also over $\Q$. Then, by~\eqref{cfrt1} and~\eqref{cfrt4}, $\K=\Q(\zeta, \beta)$ is the field of degree $3s$ with signature $(s,s)$ for which the inequality $m(\K) \leq m(\zeta) \leq 0.946467\ldots$ holds. This implies the result, since
there are infinitely many such $\be$.
\end{proof}

\section{Concluding questions}

In conclusion, we pose some questions related to Lemmas~\ref{lem:d and N} and~\ref{lem:nece} and als 
to Theorems~\ref{kiy} and~\ref{thm:sextic} which might be of interest.

\begin{quest}
\label{quest:ma<1}
 Is it true that $m(\alpha) < 1$ for a non-zero algebraic integer $\alpha$ is only possible when $\alpha$ is an algebraic unit?
 \end{quest}

\begin{quest}
\label{quest:min m}
 Is $$m(\zeta) = 
\frac{\zeta^2+\zeta^{-1}}{2}=0.946467\ldots,$$ where 
$\zeta=0.826031\ldots$ is the root of $x^6+x^2-1=0$, the smallest value of $m(\al)$ when $\al$ runs through all non-zero algebraic integers?
\end{quest}

The authors, on the basis of some very limited numerical tests,  believe that the answers to Questions~\ref{quest:ma<1} and~\ref{quest:min m}
are positive. Unfortunately exhaustive testing of these two and similar questions is 
infeasible. For example, for Question~\ref{quest:ma<1}, 
by Lemma~\ref{lem:nece}, one has to work with polynomials of degree $n \ge 24$. 

We remark that, by Theorem~\ref{kiy}, the inequality~\eqref{derf1} of Lemma~\ref{lem:d and N} 
on the absolute size
$$
\|\alpha\|^2 > n 2^{s/n-1}
$$
is sharp   for  $s=2$. The authors are unaware of any other 
positive integers $s$ for which it is 
sharp on the set of algebraic integers $\al$ of degree $n$ with $s$ real conjugates.

\begin{quest}
\label{quest:sharp} Is there a positive integer $s \ne 2$ such that for any 
$\eps>0$ there exists a non-zero algebraic integer $\alpha$ of degree $n$ (depending on $s$ and $\eps$) with exactly $s$ real conjugates for which the inequality
$$\|\alpha\|^2 < n 2^{s/n-1}+\eps$$
holds?
\end{quest}

By~\eqref{kiol1}, \eqref{susm} and~\eqref{kiol3}, if such an algebraic unit of degree $n=s+2t$ exists, it should have $s$ real conjugates close to the points
$\pm 2^{-t/n}$ and $2t$ complex conjugates close to the circle $|z|=2^{s/(2n)}$. Some restrictions on the number of real conjugates $s$ of an algebraic unit $\al$ whose $2t$
complex conjugates all lie {\it exactly} on the circle $|z|=R$ have been recently described in
\cite[Theorem~1]{new} in an entirely different context.

\section*{Acknowledgements} 
The authors would like to thank the referee for very careful reading and valuable comments. 
The research of the second and third authors was supported by the Australian
Research Council Grant DP130100237.

\end{document}